\documentclass[11pt]{article}
\usepackage{comment}
\usepackage{enumerate}
\usepackage{geometry} 
\usepackage{amsmath}
\usepackage{amsthm}
\usepackage{thmtools}
\usepackage{amssymb}
\usepackage{mathrsfs}
\usepackage{lipsum}
\usepackage{graphicx}
\graphicspath{{./figure/}}
\DeclareGraphicsExtensions{.pdf,.jpeg,.png,.jpg}
\usepackage{xcolor}
\usepackage{tikz}
\usepackage{tkz-euclide}

\usepackage{hyperref}
\usepackage[capitalize]{cleveref} 
\usepackage{titletoc}
\usepackage{float}
\newtheoremstyle{noparen}
  {3pt}{3pt}        
  {\itshape}        
  {}                
  {\bfseries}       
  {}               
  { }               
  {\thmname{#1}\thmnumber{ #2} \thmnote{\normalfont #3}}

\theoremstyle{noparen}

\numberwithin{equation}{section}

\newtheorem{theorem}{Theorem}[section]
\newtheorem{lemma}[theorem]{Lemma}
\newtheorem{proposition}[theorem]{Proposition}

\newtheoremstyle{nopunct}
  {3pt}{3pt}        
  {}                
  {}                
  {\bfseries}       
  {}                
  { }               
  {\thmname{#1}\thmnumber{ #2} \thmnote{ #3}}
\theoremstyle{nopunct}

\newtheorem{remark}[theorem]{Remark}

\newtheorem{definition}[theorem]{Definition}

\newtheoremstyle{withcitation}
{\topsep}{\topsep}{\itshape}{}{\bfseries}{.}{ }{\thmname{#1}\thmnumber{ #2}\thmnote{ (#3)}}
\theoremstyle{withcitation}

\crefname{lemma}{Lemma}{Lemma}

\newcommand{\refcond}[1]{%
  \hyperref[cond:#1]{(#1)}
} 

\title{Combinatorial Ricci Flows and Hyperbolic Structures on a Class of Compact $3$-Manifolds with Boundary}
\author{Xinrong Zhao}
\date{}
\begin{document}
\maketitle

\begin{abstract}
    In this paper, we study a combinatorial Ricci flow on closed pseudo $3$-manifolds $(M,\mathcal{T})$. We prove that if every edge in the triangulation $\mathcal{T}$ has valence at least $9$, then the combinatorial Ricci flow converges exponentially fast to the unique zero-curvature hyper-ideal metric. As a consequence, for any compact $3$-manifold $N$ with boundary admitting an ideal triangulation $\mathcal{T}_N$ whose edges all have valence at least $9$, there exists a unique complete hyperbolic metric with totally geodesic boundary on $N$ such that $\mathcal{T}_N$ is isotopic to a geometric decomposition of $N$. This provides a partial solution to the conjecture of Costantino, Frigerio, Martelli and Petronio, and hence an affirmative answer to Thurston's geometric ideal triangulation conjecture for such manifolds. Moreover, we obtain explicit upper and lower bounds for the resulting hyperbolic metric.
\end{abstract}

\noindent \textbf{Keywords:} Combinatorial Ricci flows, Hyperbolic 3-manifolds, Hyper-ideal tetrahedra

\section{Introduction}
According to Thurston’s geometrization program, the vast majority of $3$-manifolds are hyperbolic. For cusped hyperbolic $3$-manifolds, a general construction method initiated by Thurston is to obtain them by isometrically gluing ideal hyperbolic tetrahedra; see, for example, \cite{thurston1980geometry, Martelli,SnapPea}. For compact hyperbolic $3$-manifolds with totally geodesic boundary, one may similarly decompose them into truncated hyperbolic tetrahedra; see, for example, \cite{MR1208308,MR2052949,Frigerio01012004}. Motivated by Thurston’s geometric ideal triangulation conjecture for cusped hyperbolic $3$-manifolds, it is conjectured that every compact hyperbolic $3$-manifold with totally geodesic boundary admits a geometric triangulation; see, for instance, \cite{KeGethurstons}. In \cite{Largethan6}, Costantino, Frigerio, Martelli, and Petronio proved that if a compact $3$-manifold $N$, with all boundary components having genus at least $2$, admits an ideal triangulation $\mathcal{T}_N$ where every edge has valence at least $6$, then $N$ admits a hyperbolic metric with totally geodesic boundary. They further conjectured that the ideal triangulation $\mathcal{T}_N$ can be realized by truncated hyperbolic tetrahedra.
Using combinatorial Ricci flow, Feng, Ge, and Hua \cite{FengGeHua} made notable progress by providing a resolution of this conjecture in the case where each edge of $\mathcal{T}_N$ has valence at least $10$. In the present paper, we further improve their result by lowering this bound to $9$.

\begin{theorem}\label{thm:Begining}
    Let $N$ be a compact $3$-manifold with boundary such that all boundary components are surfaces of genus at least $2$. If $N$ admits an ideal triangulation $\mathcal{T}_N$ with valence at least $9$ at all edges, then there exists a unique complete hyperbolic metric on $N$ with totally geodesic boundary. Furthermore, $\mathcal{T}_N$ is isotopic to a geometric decomposition of $N$.
\end{theorem}

The uniqueness of the complete hyperbolic metric on $N$ with totally geodesic boundary follows from a standard application of Mostow's Rigidity Theorem; see \cite{Thurston_1982} for details. Thus, the main focus of this paper is to establish the existence of such a hyperbolic metric. Our approach is to find a geometric decomposition of $N$ via the combinatorial Ricci flow method, which is a large program initiated by Luo \cite{LuoFlowWithBoundary}. For other variants of combinatorial Ricci flow on $3$-manifolds, we refer the reader to \cite{MR4498159,feng2025hyperbolization,KeGethurstons}.

We now introduce our settings. Let $\mathscr{T}=\bigsqcup_{i=1}^n \sigma_i$, $n\in\mathbb{N}$, be the disjoint union of a finite collection of tetrahedra, which is naturally a simplicial complex. Suppose that the $4n$ codimension-1 faces of these tetrahedra are partitioned into $2n$ pairs, and each pair is glued together by an affine homeomorphism.  Denote the resulting quotient complex by $\mathcal{T}$ and its underlying space by $M = |\mathcal{T}|$. The pair $(M, \mathcal{T})$ is called a \emph{closed pseudo $3$-manifold}. Note that simplices in $\mathcal{T}$ are equivalence classes of simplices in $\mathscr{T}$. 
Let $V = V(\mathcal{T})$ and $E = E(\mathcal{T})$ denote the sets of vertices and edges of $\mathcal{T}$, respectively. 
For an edge $e \in E$, we define its \emph{valence} $v(e)$ to be the number of edges in $\mathscr{T}$ that belong to the equivalence class corresponding to $e$.
\begin{figure}
        \centering 
        \includegraphics[width=0.6\linewidth]{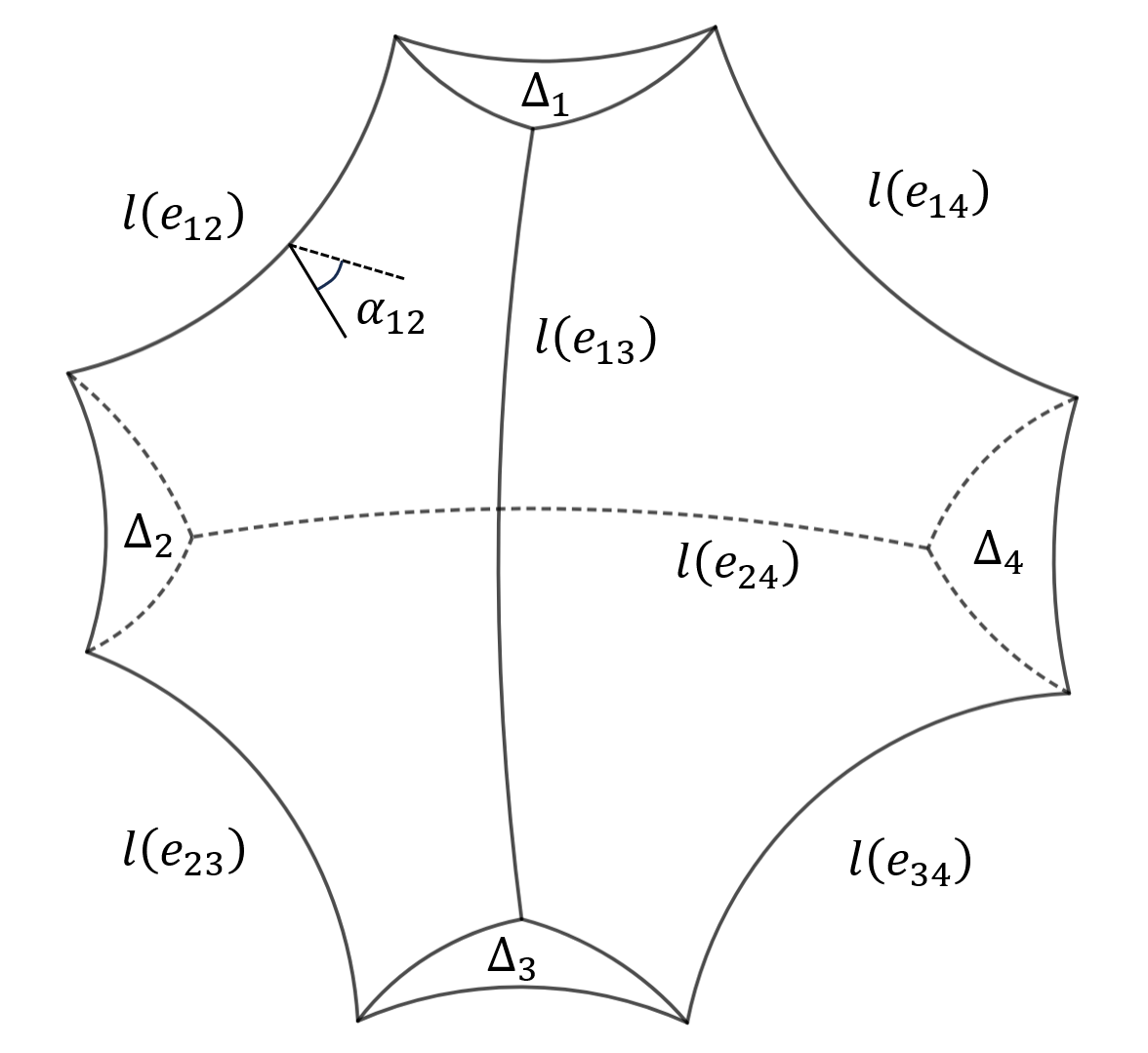}
        \caption{Hyper-ideal tetrahedron.}
        \label{fig:Tetrahedron}
    \end{figure}

    A hyper-ideal tetrahedron $\sigma$ in hyperbolic $3$-space $\mathbb{H}^3$ is a compact convex polyhedron diffeomorphic to a truncated tetrahedron in the Euclidean space $\mathbb{R}^3$ such that all its four hexagonal faces are right-angled hyperbolic hexagons; see Figure \ref{fig:Tetrahedron}. The four triangular faces $\Delta_i$, $i=1,2,3,4$, are called vertex triangles. By definition, the dihedral angle formed by a hexagonal face and a vertex triangle is $\frac{\pi}{2}$. An edge in a hyper-ideal tetrahedron is the intersection of two hexagonal faces. For each pair $\{i,j\}\subset\{1,2,3,4\}$ we denote by $e_{ij}$ the edge connecting the vertex triangles $\Delta_i$ and $\Delta_j$. The dihedral angle at the edge $e_{ij}$, denoted by $\alpha_{ij}$, is the angle formed by the two hexagonal faces adjacent to $e_{ij}$, which lies in $(0,\pi)$. The length of the edge $e_{ij}$ is denoted by $l_{ij}$. It is known that the geometry of a hyper-ideal tetrahedron is determined by the edge length vector $l=(l_{12},l_{13},l_{14},l_{23},l_{24},l_{34})\in\mathbb{R}^6_{>0}$. Therefore, the space of isometry classes of hyper-ideal tetrahedra can be identified with a subset $\mathcal{L}\subset\mathbb{R}^6$. Let $A$ be a finite set. We write $\mathbb{R}^A$ (resp. $\mathbb{R}^A_{>0}$) for the space of all real-valued (resp. positive real-valued) functions on $A$. After fixing an ordering of $A$, the space $\mathbb{R}^A$ is naturally identified with the Euclidean space $\mathbb{R}^{|A|}$, where $|A|$ denotes the cardinality of $A$.

    Combinatorially, each hyper‑ideal tetrahedron corresponds to a tetrahedron $\sigma$: its four vertex triangles $\Delta_i$ correspond to the four vertices $v_i$ of $\sigma$, and each edge $e_{ij}$ corresponds to the edge connecting $v_i$ and $v_j$ in $\sigma$. 
    For a closed pseudo $3$-manifold $(M,\mathcal{T})$, assigning a hyper-ideal tetrahedron to each tetrahedron of $\mathcal{T}$ yields a hyper-ideal polyhedral metric on $M$.
    
\begin{definition}\label{def:HyperidealMetric}
    A hyper-ideal polyhedral metric, called hyper-ideal metric for short, on $(M,\mathcal{T})$ is obtained by replacing each tetrahedron in $\mathcal{T}$ with a hyper‑ideal tetrahedron and replacing each affine homeomorphism between two faces by an isometry between the corresponding two hexagonal faces. The resulting metric space, denoted by $S(M,\mathcal{T},l)$, is uniquely determined by the edge length vector 
    $$
    l = (l_{e_1},\dots,l_{e_m}) \in \mathbb{R}^E_{>0},
    $$
    where $E=\{e_1,\dots,e_m\}$ is the edge set of $\mathcal{T}$. We denote by 
    $$
    \mathcal{L}(M,\mathcal{T}) \subset \mathbb{R}^E_{>0}
    $$ the space of all hyper-ideal metrics on  $(M, \mathcal{T})$, parameterized by the edge‑length vector $l$.

\end{definition}

\begin{definition}
   Let $(M,\mathcal{T})$ be a closed pseudo $3$-manifold with a hyper-ideal metric determined by the edge length vector $l\in\mathcal{L}(M,\mathcal{T})$. The \emph{Ricci curvature} at an edge $e\in E$ is defined as
   $$
    K_e(l)=2\pi-C_e
   $$
   where $C_e$ is the cone angle at $e$, that is, the sum of the dihedral angles at $e$ over all tetrahedra incident to $e$. The Ricci curvature vector of $(M,\mathcal{T})$ associated with $l$ is then given by
   $$
   K=K(l)=(K_{e_1}(l),\cdots,K_{e_m}(l)),
   $$
   where $E=\{e_1,\dots,e_m\}$.
\end{definition}

Suppose that $N$ is a compact $3$-manifold with boundary such that all boundary components are surfaces of genus at least $2$. Our goal is to find a hyperbolic metric on $N$ with totally geodesic boundary.
Let $C(N)$ denote the compact space obtained by coning each boundary component of $N$ to a point. If the boundary of $N$ has $k$ components, then $C(N)$ contains exactly $k$ cone points $\{v_1, \dots, v_k\}$, and the complement $C(N) - \{v_1, \dots, v_k\}$ is homeomorphic to $N - \partial N$.
An \emph{ideal triangulation} of $N$ is a triangulation $\mathcal{T}$ of the space $C(N)$ whose vertex set is precisely $\{v_1, \dots, v_k\}$. 
In the terminology introduced earlier, $(C(N), \mathcal{T})$ is a closed pseudo $3$-manifold. If there is a hyper-ideal metric $l \in \mathcal{L}(C(N), \mathcal{T})$ such that the Ricci curvature at every edge is $0$, then the metric space $S(C(N), \mathcal{T}, l)$ constructed in \cref{def:HyperidealMetric} induces a hyperbolic metric on $N$ with totally geodesic boundary. In this case, we call $(C(N), \mathcal{T})$ a \emph{geometric decomposition} (or \emph{geometric realization}) of a hyperbolic metric on $N$ associated with the ideal triangulation $\mathcal{T}$.

To study the existence of hyper-ideal metrics with zero Ricci curvature on a closed pseudo $3$-manifold $(M, \mathcal{T})$, Luo \cite{LuoFlowWithBoundary}, motivated by \cite{2003Combinatorial}, introduced the following combinatorial Ricci flow:
\begin{equation}\label{eq:RicciFlow1}
    \frac{d}{dt} l(t) = K(l(t)), \quad\forall t \geq 0,
\end{equation}
where $l(t) \in \mathcal{L}(M,\mathcal{T})$. In components, the vector valued equation \eqref{eq:RicciFlow1} can be written as
$$
\frac{d}{dt} l_e(t) = K_e(l(t)), \quad\forall e\in E,\ t \geq 0.
$$
This is a negative gradient flow of a locally convex function on $\mathcal{L}(M,\mathcal{T})$ which is related to the co‑volume functional. A principal difficulty in applying the flow method is that the space $\mathcal{L}(M,\mathcal{T})$ is not convex in $\mathbb{R}^E$.

To address this difficulty, Luo and Yang \cite{LuoYang} considered the general framework of generalized hyper‑ideal metrics. For each vector $l=(l_{12},l_{13},l_{14},l_{23},l_{24},l_{34})\in\mathbb{R}^6_{>0}$, one can define extended dihedral angles $\widetilde{\alpha}_{ij}$ which are continuous functions of edge lengths $l_{ij}$, extending the dihedral angles $\alpha_{ij}$ of hyper-ideal tetrahedra; see \cref{def:ExtendedDihedralAngle}. We say that such a vector corresponds to a generalized hyper-ideal tetrahedron, and that a vector $l\in\mathbb{R}^6_{>0}-\mathcal{L}$ corresponds to a degenerate hyper-ideal tetrahedron. 
A \emph{generalized hyper‑ideal metric} on a closed pseudo $3$-manifold $(M,\mathcal{T})$ is then defined by an edge length vector $l\in\mathbb{R}^E_{>0}$. It is a hyper-ideal metric if and only if $l\in\mathcal{L}(M,\mathcal{T})$. Although it may not correspond to a metric‑space structure, the extended dihedral angles remain well-defined, which is sufficient for the definition of the flow.

Let $(M, \mathcal{T})$ be a closed pseudo $3$-manifold, which is the quotient space of the simplicial complex $\mathscr{T} = \bigsqcup_{i=1}^n  \sigma_i$. Denote by $E(\cdot)$ and $T(\cdot)$ the sets of edges and tetrahedra of a given complex, respectively. For an edge $e \in E(\cdot)$ and a tetrahedron $\sigma \in T(\cdot)$, we write $e \sim \sigma$ if $e$ is incident to $\sigma$, i.e., if $e$ is contained in $\sigma$.
Let $P_E: E(\mathscr{T}) \to E(\mathcal{T})$ and $P_T: T(\mathscr{T}) \to T(\mathcal{T})$ be the projection maps (or the quotient maps) on the set of edges and tetrahedra, respectively. Recall that $E=E(\mathcal{T})$. It is clear that the valence $v(e)$ of an edge $e\in E$ is the number of edges in $P_E^{-1}(e)$.
Let $l \in \mathbb{R}^{E}_{>0}$ be a generalized hyper‑ideal metric on $(M, \mathcal{T})$. This induces a length vector $\hat{l} := l \circ P_E: E(\mathscr{T}) \to \mathbb{R}_{>0}$ on $E(\mathscr{T})$. We can then endow each $\sigma_i \in \mathscr{T}$ with the structure of a generalized hyper-ideal tetrahedron using the length vector $\hat{l}$. For each $\hat{e} \in E(\mathscr{T})$, it is contained in a unique tetrahedron $\sigma_j$. We denote by $\alpha(\hat{e})$ the extended dihedral angle of $\hat{e}$ in $\sigma_j$ with respect to the length vector $\hat{l}$.

\begin{definition}
Let $l \in \mathbb{R}^{E}_{>0}$ be a generalized hyper-ideal metric on $(M, \mathcal{T})$. The \emph{generalized Ricci curvature} of an edge $e \in E$ is defined as
$$
\widetilde{K}_e(l) = 2\pi - \sum_{\hat{e} \in P_E^{-1}(e)} \alpha(\hat{e}).
$$
\end{definition}

By introducing a change of variables in \eqref{eq:RicciFlow1}, Feng, Ge, and Hua \cite{FengGeHua} studied the following extended Ricci flow defined on the space of generalized hyper‑ideal metrics $\mathbb{R}^E_{>0}$:
\begin{equation}\label{eq:ExtendedRicciFlow}
\begin{cases}
\dfrac{d}{dt} \, l(t) = \widetilde{K}\bigl(l(t)\bigr) \, l(t), & l(t) \in \mathbb{R}^E_{>0},\ \forall t > 0, \\
l(0)=l^0 \in \mathbb{R}^E_{>0}.
\end{cases}
\end{equation}
A solution $l(t)$ of \eqref{eq:ExtendedRicciFlow} is said to \emph{converge} if there exists $l \in \mathbb{R}^E_{>0}$ such that
$$
l(t) \to l, \qquad t \to \infty.
$$
They proved the long‑time existence and uniqueness of solutions of the extended Ricci flow; see \cref{sec:ExtendedRicciFlow} for details.

The following is the main result of the present paper.

\begin{theorem}\label{thm:MainTheorem}
    Let $(M,\mathcal{T})$ be a closed pseudo $3$-manifold such that $v(e) \ge 9$ for every edge $e \in E$. Then there exists a zero-curvature hyper-ideal metric $l \in \mathcal{L}(M,\mathcal{T})$, which is unique in $\mathbb{R}^{E}_{>0}$. Moreover, for each $e\in E$,
    $$
        l_e \in \left[\operatorname{arccosh}\gamma_{v(e)}, \operatorname{arccosh} b_{v(e)}\right]\subset(0,\operatorname{arccosh}2],
        $$
    where $\gamma_{v(e)},b_{v(e)} > 1$ are constants depending only on the valence of the edge $e$, defined in \eqref{eq:GammanExpression} and \eqref{eq:bnExpression}.
    In this case, for all initial data $l^{0} \in \mathbb{R}^{E}_{>0}$, the solution $l(t)$ of the extended Ricci flow \eqref{eq:ExtendedRicciFlow} converges to $l$ exponentially fast as $t \to \infty$.
\end{theorem}

\begin{remark}\label{rmk:MainTheorem}
\begin{enumerate}
    \item[(1)] Suppose that $N$ is a compact $3$-manifold with boundary admitting an ideal triangulation $\mathcal{T}_N$ in which every edge has valence at least $9$. Set $M=C(N)$ and $\mathcal{T}=\mathcal{T}_N$, and let $l$ be the zero-curvature hyper-ideal metric obtained in \cref{thm:MainTheorem}. It follows that the metric space $S(M,\mathcal{T},l)$ is a compact hyperbolic $3$-manifold with totally geodesic boundary, which implies the existence part of \cref{thm:Begining}.
    
    \item[(2)] Compared with \cite[Theorem~1.9]{FengGeHua}, \cref{thm:MainTheorem} yields sharper quantitative bounds for the zero-curvature hyper-ideal metric. Numerical estimates of $\gamma_n$ and $b_n$ for $n\leq 50$ are given in \cref{tab:EstimatedBounds}.
    
\end{enumerate}
\end{remark}

\paragraph{Strategy of the proof.}
The central idea of this paper is to show that, for suitable initial data $l^0\in (0,\operatorname{arccosh}2)^E$, the flow remains in the region $(0,\operatorname{arccosh}2]^E$ for all $t\geq0$. By previous results in \cite{FengGeHua}, the flow then always lies in $\mathcal{L}(M,\mathcal{T})$ and converges exponentially fast to a zero-curvature hyper-ideal metric.
Our proof proceeds by contradiction. Suppose that the flow eventually leaves the region $(0,\operatorname{arccosh}2]^E$, and let $t_0>0$ be the first time at which there exists an edge $e_0\in E$ such that $l_{e_0}(t_0)=\operatorname{arccosh}2$. By the choice of $t_0$, we must have $l_{e_0}'(t_0)\geq 0$. On the other hand, the flow equation gives $l_{e_0}'(t_0) = K_{e_0}(l(t_0))l_{e_0}(t_0)$. Therefore, it suffices to prove that $K_{e_0}(l(t_0))<0$ to derive a contradiction. Since $K_{e_0}(l(t_0))$ is defined as $2\pi$ minus the sum of the dihedral angles around $e_0$ at time $t_0$, our main task reduces to showing that this angle sum must strictly exceed $2\pi$.
This is achieved by establishing a priori estimates for the edge lengths, which subsequently yield lower bounds for the corresponding dihedral angles along the flow. Let $\sigma$ be a generalized hyper-ideal tetrahedron and let $e$ be an edge of $\sigma$. Using the explicit formula for the dihedral angles, we will show in \cref{sec:PropertiesDihedralAngles} that within the region $(0,\operatorname{arccosh}2]^E$, the dihedral angle $\alpha_e$ at $e$ in $\sigma$ is monotonically decreasing with respect to the lengths of the four edges adjacent to $e$ in $\sigma$, and monotonically increasing with respect to the length of the edge opposite to $e$ in $\sigma$. Consequently, sharper upper bounds for the adjacent edge lengths and sharper lower bounds for the opposite edge length produce a tighter lower bound for the dihedral angle at $e$ in $\sigma$. To this end, we first establish preliminary a priori bounds for the edge lengths in \cref{sec:PropertiesDihedralAngles}. Next, by exploiting the combinatorial structure of the triangulation, we derive explicit upper bounds for the minimum lengths of opposite edge pairs in \cref{sec:UpperBounds}. Finally, through numerical verification, we confirm that the sum of the dihedral angles around $e_0$ at time $t_0$ strictly exceeds $2\pi$, thereby completing the proof.

The paper is organized as follows. In \cref{sec:Preliminaries}, we introduce some results of generalized hyper-ideal tetrahedra and the extended Ricci flow \eqref{eq:ExtendedRicciFlow}, as obtained by \cite{LuoYang} and \cite{FengGeHua}. In \cref{sec:PropertiesDihedralAngles}, we establish key properties of the dihedral angles in hyper-ideal tetrahedra and derive a priori bounds on the edge lengths under the extended Ricci flow. \cref{sec:UpperBounds} is devoted to refining upper bounds for these edge lengths along the flow. Finally, in \cref{sec:Proof}, we prove the main result \cref{thm:MainTheorem}.

\section{Preliminaries}\label{sec:Preliminaries}

\subsection{Generalized hyper-ideal tetrahedra}

The following is an alternative description of hyper-ideal tetrahedra given by \cite{MR3424676}. Let $\mathbb{K}^3\subset\mathbb{R}^3$ be the open ball representing the Klein model of $\mathbb{H}^3$. A hyper-ideal hyperbolic tetrahedron can be constructed as follows: Let $\mathcal{P}\subset\mathbb{R}^3$ be a convex Euclidean tetrahedron such that all four vertices $v_i$ $(i=1,2,3,4)$ of $\mathcal{P}$ lie in $\mathbb{R}^3- \overline{\mathbb{K}^3}$ and all edges of $\mathcal{P}$ have a nonempty intersection with $\mathbb{K}^3$. Let $\mathcal C_i$ denote the cone with apex $v_i$ that is tangent to the boundary $\partial \mathbb{K}^3$. Let $\pi_i$ be the half-space not containing $v_i$ such that $\partial \pi_i \cap \partial \mathbb{K}^3 = \mathcal C_i \cap \partial \mathbb{K}^3$. The truncated tetrahedron $P:= \mathcal{P} \cap \bigcap_i \pi_i$ is then a hyper-ideal tetrahedron in hyperbolic space, as shown in Figure \ref{fig:Klein}.

\begin{figure}
    \centering
    \vspace{-10pt} 
    \includegraphics[width=0.7\linewidth]{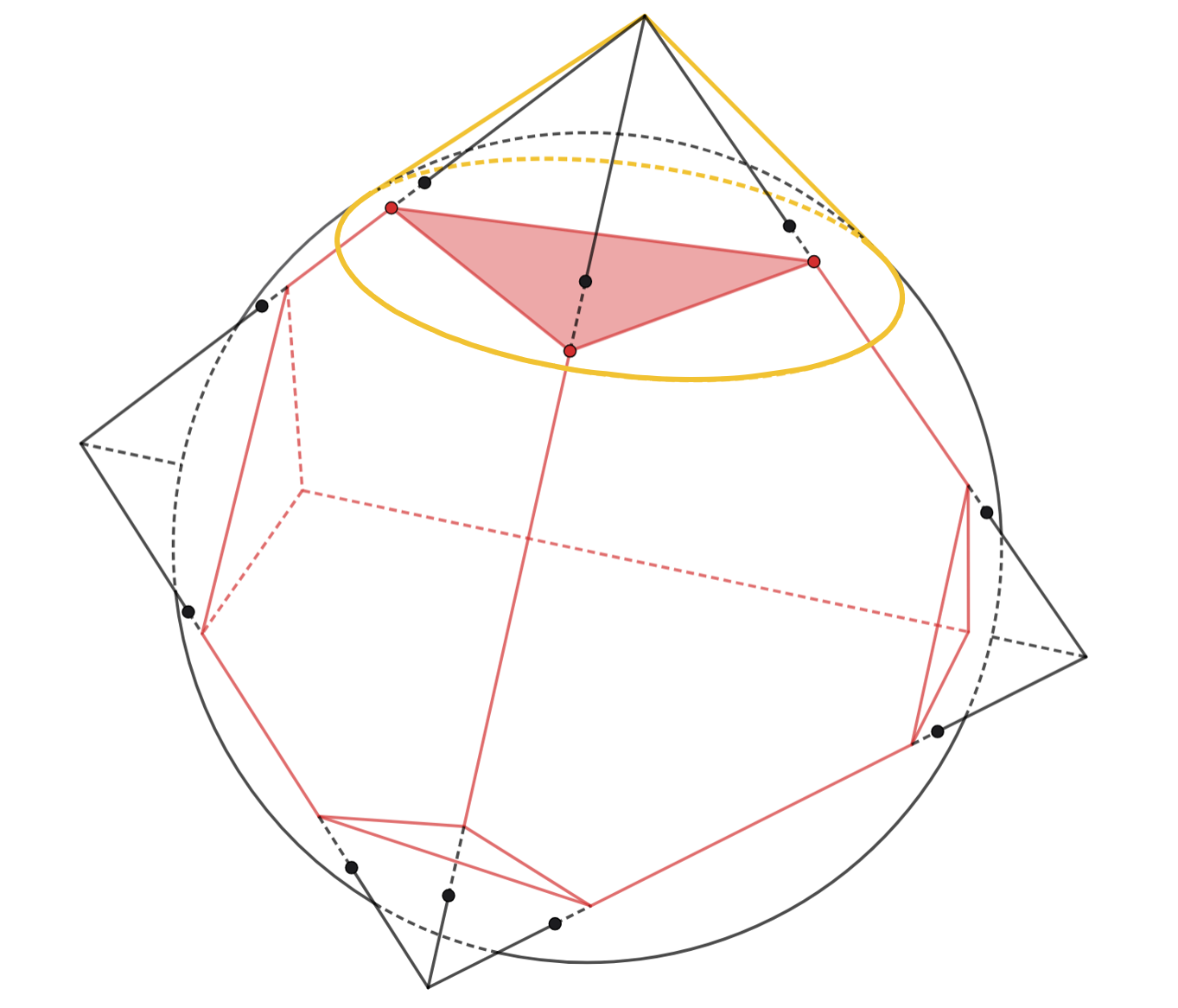}
    \vspace{-10pt} 
    \caption{A hyper-ideal tetrahedron in the Klein model.}
    \label{fig:Klein}
\end{figure}

We now recall several key results that characterize the space of hyper‑ideal tetrahedra. Details can be found in \cite{LuoYang}. Consider a combinatorial tetrahedron with vertex set $\{1,2,3,4\}$ and a hyper‑ideal tetrahedron $\sigma$ based on it. We adopt the notation that $\Delta_i$ ($e_{ij}$, $l_{ij}$ and $\alpha_{ij}$, resp.) represents a vertex triangle (an edge, the edge length of $e_{ij}$ and a dihedral angle, resp.) as in the introduction, and denote by $H_{ijk}$ the hexagonal face adjacent to $e_{ij}$, $e_{ik}$ and $e_{jk}$, where $\{i,j,k\}\subset\{1,2,3,4\}$.

\begin{proposition}[{\cite{MR1943885,MR1075164}}]
Let $\sigma$ be a hyper-ideal tetrahedron. Then
\begin{enumerate}
    \item[\rm(1)] The isometry class of $\sigma$ is determined by its dihedral angle vector $(\alpha_{12},\dots,\alpha_{34})$ in $\mathbb{R}^6$, which satisfies the condition that $\alpha_{ij}>0$ for all $\{i,j\}\subset\{1,2,3,4\}$, and $\sum_{j\neq i}\alpha_{ij}<\pi$ for all $1\leq i\leq4$.
    \item[\rm(2)] Conversely, given any $(\alpha_{12},\dots,\alpha_{34})\in\mathbb{R}_{>0}^6$ satisfying that $\sum_{j\neq i}\alpha_{ij}<\pi$ for each $1\leq i\leq4$, there exists a hyper-ideal tetrahedron whose dihedral angles are given by $\alpha_{ij}$.
    \item[\rm(3)] The isometry class of $\sigma$ is also determined by its edge length vector $(l_{12},\dots,l_{34})\in\mathbb{R}_{>0}^6$.
\end{enumerate}
\end{proposition}
Thus, the space of isometry classes of hyper‑ideal tetrahedra, when parametrized by dihedral angles, is the open convex polytope
\begin{equation}
    \mathscr{B} = \Bigl\{ (\alpha_{12},\dots,\alpha_{34}) \in \mathbb{R}_{>0}^6 \ \Big|\ 
    \sum_{j\neq i}\alpha_{ij} < \pi \ \text{for all } i=1,2,3,4 \Bigr\}.
\end{equation}

To define the extended dihedral angles, we first recall the formula expressing dihedral angles in terms of edge lengths for hyper‑ideal tetrahedra; see \cite[Lemma 4.3]{LuoYang} for details.

\begin{lemma}
Let $l=(l_{12},\dots,l_{34})\in\mathcal{L}$ be an edge length vector representing a hyper-ideal tetrahedron and let $\{i,j,k,h\}=\{1,2,3,4\}$. Set $x_{ij}=\cosh l_{ij}$, with the convention that $l_{ij}=l_{ji}$ for $i\neq j$. Define the function $\phi_{ij}:\mathcal{L}\to\mathbb{R}$ by
\begin{equation}\label{eq:phiFormula}
\phi_{ij}=\frac{x_{ik}x_{ih}+x_{jk}x_{jh}+x_{ij}x_{ik}x_{jh}+x_{ij}x_{ih}x_{jk}-x_{ij}^2x_{kh}+x_{kh}}{\sqrt{2x_{ij}x_{ik}x_{jk}+x_{ij}^2+x_{ik}^2+x_{jk}^2-1}\sqrt{2x_{ij}x_{ih}x_{jh}+x_{ij}^2+x_{ih}^2+x_{jh}^2-1}},
\end{equation}
Then the dihedral angle at the edge $e_{ij}$ is given by $\alpha_{ij}=\arccos(\phi_{ij})$.
\end{lemma}

Recall that every $l\in\mathbb{R}^6_{>0}$ corresponds to a generalized hyper-ideal tetrahedron. If $l\in\mathcal{L}$, then the tetrahedron is a hyper-ideal tetrahedron. Otherwise, if $l\in\mathbb{R}_{>0}^6-\mathcal{L}$, then the tetrahedron is a degenerate hyper-ideal tetrahedron. More generally, by allowing edge lengths to be zero, one obtains a larger class of objects parametrized by $\mathbb{R}_{\ge 0}^6$, which we refer to as \emph{generalized hyper‑ideal tetrahedra in the wide sense}.

\begin{definition}\label{def:ExtendedDihedralAngle}
Let $l \in \mathbb{R}_{\ge 0}^6$ represent a generalized hyper‑ideal tetrahedron in the wide sense. For any distinct vertices $\{i,j\} \subset \{1,2,3,4\}$, let $\phi_{ij}$ be given by \eqref{eq:phiFormula} and define
$$
\alpha_{ij} := \arccos\left(\max\{-1,\; \min\{\phi_{ij},\,1\}\}\right).
$$
We call $\alpha_{ij}$ the \emph{extended dihedral angle} at the edge $e_{ij}$.
\end{definition}
Note that the extended dihedral angle $\alpha_{ij}$ was denoted by $\widetilde{\alpha}_{ij}$ in the introduction in order to distinguish it from the usual dihedral angle. In the remainder of this paper, we shall simply call it the dihedral angle when no confusion arises.
It is evident that both $\phi_{ij}$ and $\alpha_{ij}$ are continuous functions on $\mathbb{R}_{\ge 0}^6$. Furthermore, for a hyper-ideal tetrahedron, $\alpha_{ij}$ coincides with the usual dihedral angle at the edge $e_{ij}$.

The space $\mathcal{L}\subset\mathbb{R}^6$ can be characterized by the range of $\phi_{ij}$ as follows.

\begin{proposition}[{\cite[Proposition 4.4 and Lemma 4.7]{LuoYang}}]\label{prop:L-description}
The space of hyper‑ideal tetrahedra parametrized by edge lengths is given by
$$
\mathcal{L} = \big\{ l \in \mathbb{R}_{\ge 0}^6 \ \big|\ 
\phi_{ij}(l) \in (-1,1) \ \text{for all } \{i,j\}\subset\{1,2,3,4\} \big\}.
$$
\end{proposition}

\subsection{Co-volume of a generalized hyper-ideal tetrahedron}

Let $\operatorname{vol}(\cdot)$ denote the hyperbolic volume of a hyper-ideal tetrahedron. Following Casson-Rivin's angle structure theory (see, for example, \cite{MR1283870,MR1985831,LuoAngleStructure,MR2369192}), the volume can be naturally regarded as a real-valued function on the set $\mathscr{B}$ of dihedral angles. It satisfies the Schl\"afli formula (see \cite{MR1678473} for a more general setting)
\begin{equation}\label{eq:Schlafli}
\frac{\partial \operatorname{vol}}{\partial \alpha_{ij}} = -\frac{1}{2}l_{ij}.
\end{equation}

Let $\mathcal{L} \subset \mathbb{R}_{>0}^6$ be the space of hyper-ideal tetrahedra parametrized by their edge length vector $l = (l_{12}, \dots, l_{34})$. The Legendre transform of the volume, known as the \emph{co-volume functional}, is defined by
\begin{equation}
\operatorname{cov}(l) = 2 \operatorname{vol}(l) + \sum_{i<j} \alpha_{ij} l_{ij},\quad l\in\mathcal{L}.
\end{equation}

\begin{proposition}[{\cite[Corollary 5]{LuoFlowWithBoundary}}]
The functional $\operatorname{cov}: \mathcal{L} \to \mathbb{R}$ is smooth and locally strictly convex.
\end{proposition}

The form of the co-volume functional can be dated back to \cite{MR1106755,MR1815214,MR3375525}.
By the Schl\"afli formula, the gradient of the co-volume functional is given by 
\begin{equation}\label{eq:cov-PartialDifferential}
    \frac{\partial \operatorname{cov}}{\partial l_{ij}} = \alpha_{ij}.
\end{equation}
Thus, the differential 1-form $\omega = \sum_{i<j} \alpha_{ij} dl_{ij}=d\operatorname{cov}$ is closed on $\mathcal{L}$, and $\operatorname{cov}(l)$ can be recovered via integration of $\omega$ (see \cite{LuoFlowWithBoundary}).

Since $\mathcal{L}$ is not a convex subset of $\mathbb{R}^6$, $\operatorname{cov}$ may not be globally convex. To address this, Luo and Yang \cite{LuoYang} extended $\operatorname{cov}$ to a $C^1$-smooth and convex function on the entire space $\mathbb{R}^6$. For each $l=(l_{12}, \dots, l_{34}) \in \mathbb{R}^6$, let $l^+ = (l_{12}^+, \dots, l_{34}^+)\in\mathbb{R}^6_{\geq0}$ where $l_{ij}^+ = \max\{l_{ij}, 0\}$. We extend the dihedral angle functions to $\mathbb{R}^6$ by setting 
$$
\alpha_{ij}(l) = \alpha_{ij}(l^+), \quad\forall l\in\mathbb{R}^6,
$$
where $\alpha_{ij}(l^+)$ is the extended dihedral angle defined in \cref{def:ExtendedDihedralAngle}. This defines a continuous 1-form $\mu$ on $\mathbb{R}^6$ by
\begin{equation}\label{eq:mu-form}
\mu(l) = \sum_{i<j} \alpha_{ij}(l) dl_{ij}.
\end{equation}

\begin{proposition}[{\cite[Proposition 4.10]{LuoYang}}]
The continuous differential 1-form 
$$
\mu(l)=\sum_{i< j}\alpha_{ij}(l)dl_{ij}
$$
is closed in $\mathbb{R}^6$, that is, for each Euclidean triangle $\Delta$ in $\mathbb{R}^6$, $\int_{\partial\Delta}\mu=0$.
\end{proposition}

We define the extended co-volume functional $\operatorname{cov}:\mathbb{R}^6\to\mathbb{R}$ by the integral
\begin{equation}\label{eq:cov-def}
\operatorname{cov}(l)=\int_{(0,\dots,0)}^{l}\mu+\operatorname{cov}(0,\dots,0),
\end{equation}
where $\operatorname{cov}(0,\dots,0)=2\operatorname{vol}(0,\cdots , 0) = 16\Lambda\left(\frac{\pi}{4}\right)$ by the result of Ushijima \cite{Ushijima}. Here $\Lambda(a)=-\int_0^a\ln|2\sin t|\,dt$ is the Lobachevsky function. The same notation $\operatorname{cov}$ is retained as the extended functional coincides with the original co‑volume functional on $\mathcal{L}$.

\begin{proposition}[{\cite[Corollary 4.12]{LuoYang}}]\label{prop:CoVolumeConvex}
The extended co-volume functional $\operatorname{cov}:\mathbb{R}^6\to\mathbb{R}$ is a $C^1$-smooth convex function.
\end{proposition}

\subsection{Extended Ricci flows}\label{sec:ExtendedRicciFlow}
Let $(M,\mathcal{T})$ be a closed pseudo $3$-manifold, which is the quotient space of the simplicial complex $\mathscr{T} = \bigsqcup_{i=1}^n  \sigma_i$. We denote the set of edges by $E(\mathcal{T}) = \{e_1, \dots, e_m\}$, where $m$ is the number of edges in $\mathcal{T}$. The edge length vector is then given by
$$
l=(l_{e_1},l_{e_2},\cdots,l_{e_m}).
$$
Given a generalized hyper-ideal metric $l \in \mathbb{R}^E_{>0}$, the generalized Ricci curvature at edges is described by the vector
$$
\widetilde{K}(l)=\bigl(\widetilde{K}_{e_1}(l),\widetilde{K}_{e_2}(l),\dots,\widetilde{K}_{e_m}(l)\bigr).
$$
When $l$ is a hyper-ideal metric, i.e., $l\in \mathcal{L}(M,\mathcal{T})$, we simply write $K(l)$ instead of $\widetilde{K}(l)$.

Following Luo and Yang \cite{LuoYang}, the \emph{co\-volume functional} on a closed pseudo $3$--manifold $(M,\mathcal{T})$ is defined by
$$
\operatorname{cov}:\mathbb{R}^E\to\mathbb{R},\quad 
\operatorname{cov}(l)=\sum_{\hat{\sigma}\in T(\mathscr{T})}\operatorname{cov}_{\hat{\sigma}}\left(\hat{l}\big|_{E(\hat{\sigma})}\right),
$$
where $\operatorname{cov}_{\hat{\sigma}}$ is the co-volume functional of the tetrahedron $\hat{\sigma}\in T(\mathscr{T})$ defined in \eqref{eq:cov-def}. 
We further introduce the following functional:
\begin{definition}\label{def:H-functional}
Define $H:\mathbb{R}^E\to\mathbb{R}$ by
$$
H(l)=\operatorname{cov}(l)-2\pi\sum_{e\in E}l_e,\quad \forall l\in\mathbb{R}^E.
$$
\end{definition}
By \cref{prop:CoVolumeConvex}, each $\operatorname{cov}_{\hat{\sigma}}$ is a $C^1$-smooth convex function on $\mathbb{R}^6$ and is locally strictly convex on $\mathcal{L}$. Consequently, the global functional $H$ inherits these properties.

\begin{proposition}\label{prop:H-Convex}
The functionals $\operatorname{cov}, H:\mathbb{R}^E\to\mathbb{R}$ are $C^1$-smooth and convex on $\mathbb{R}^E$. Moreover, they are smooth and locally strictly convex on $\mathcal{L}(M,\mathcal{T})$.
\end{proposition}

Let $\hat{\sigma}\in T(\mathscr{T})$ be a tetrahedron with the generalized hyper-ideal metric given by $\hat{l}$. By \eqref{eq:cov-PartialDifferential}, for each edge $\hat{e}\sim\hat{\sigma}$,
$$
\frac{\partial\operatorname{cov}_{\hat{\sigma}}}{\partial \hat{l}_{\hat{e}}}=\alpha(\hat{e}).
$$
Hence, for each $e_j\in E$,
\begin{equation}\label{eq:H-PartialDerivative}
    \begin{aligned}
\frac{\partial H}{\partial l_{e_j}}
&=\frac{\partial\operatorname{cov}(l)}{\partial l_{e_j}}-2\pi
   =\frac{\partial}{\partial l_{e_j}}\Bigl(\sum_{\hat{\sigma}\in T(\mathscr{T})}\operatorname{cov}_{\hat{\sigma}}\left(\hat{l}\big|_{E(\hat{\sigma})}\right)\Bigr)-2\pi\\
&=\sum_{\hat{\sigma}\in T(\mathscr{T})}\sum_{\hat{e}\in P_E^{-1}(e_j)\cap E(\hat{\sigma})}\frac{\partial\operatorname{cov}_{\hat{\sigma}}}{\partial\hat{l}_{\hat{e}}}\frac{\partial\hat{l}_{\hat{e}}}{\partial l_{e_j}}-2\pi\\
&=\sum_{\hat{e}\in P_E^{-1}(e_j)}\alpha(\hat{e})-2\pi
   =-\widetilde{K}_{e_j}(l),
\end{aligned}
\end{equation}
where we have used $\hat{l}_{\hat{e}}=l_{P_E(\hat{e})}=l_{e_j}$ for each $\hat{e}\in P_E^{-1}(e_j)$. In particular, for any $l\in \mathcal{L}(M,\mathcal{T})$,
\begin{equation}\label{eq:H-ordinary-PartialDerivative}
\frac{\partial H}{\partial l_{e_j}}=-K_{e_j}(l)\quad \forall e_j\in E.
\end{equation}
This shows that the extended Ricci flow \eqref{eq:ExtendedRicciFlow} is a variant of the negative gradient flow associated with the functional $H$. It follows from \eqref{eq:H-PartialDerivative} that along any solution $l(t)$ of the extended Ricci flow,
$$
\frac{dH(l(t))}{dt}
= \sum_{e_j\in E} \frac{\partial H}{\partial l_{e_j}}\,\frac{dl_{e_j}}{dt}
= -\sum_{e_j\in E} \widetilde{K}_{e_j}^{\,2}\bigl(l(t)\bigr)\,l_{e_j}(t) \leq 0,
$$
which yields the following proposition.
\begin{proposition}\label{prop:H-non-ecreasing}
The functional $H$ is non‑increasing along the extended Ricci flow \eqref{eq:ExtendedRicciFlow}.
\end{proposition}

We now present several results regarding the extended Ricci flow, as established in \cite{FengGeHua}.
\begin{theorem}[{\cite[Theorem 1.5]{FengGeHua}}]\label{thm:existence-uniqueness}
Let $(M,\mathcal{T})$ be a closed pseudo $3$-manifold. For any generalized hyper‑ideal metric $l^0 \in \mathbb{R}^E_{>0}$, the extended Ricci flow \eqref{eq:ExtendedRicciFlow} admits a unique solution $l(t)$ defined for all time $t \in [0,\infty)$.
\end{theorem}

\begin{theorem}[{\cite[Theorem 1.6]{FengGeHua}}]\label{thm:convergence-equivalence}
Let $(M,\mathcal{T})$ be a closed pseudo $3$-manifold. There exists a zero‑curvature hyper‑ideal metric on $(M,\mathcal{T})$ if and only if the extended Ricci flow \eqref{eq:ExtendedRicciFlow} converges to a hyper‑ideal metric for some initial data $l^0 \in \mathbb{R}^E_{>0}$.
In this case, the extended Ricci flow converges exponentially fast to a zero‑curvature hyper‑ideal metric for any initial data $l^0 \in \mathbb{R}^E_{>0}$.
\end{theorem}
Given an interval $I \subset (0,\infty)$, we write
$$
I^{E} \;:=\; \bigl\{ l \in \mathbb{R}^{E}_{>0} \;:\; l(e) \in I \ \ \forall e \in E \bigr\}.
$$
\begin{proposition}\label{prop:BoundnessToExistence}
    Let $(M, \mathcal{T})$ be a closed pseudo $3$-manifold and $\{l(t) \mid t \in [0, \infty)\} \subset \mathbb{R}^E_{>0}$ be the solution of the extended Ricci flow with initial data $l^0\in\mathbb{R}_{>0}^E$. Suppose that there exist constants $c,C>0$ such that 
    \begin{equation}\label{eq:BoundnessOfSolution}
        l(t) \in [c,\,C]^{E}, \quad \forall t\in[0,\infty).
    \end{equation}
    Then there
    exists a generalized hyper‑ideal metric $l\in[c,C]^E$ on $(M,\mathcal{T})$ such that $\widetilde{K}_e(l)=0$ for all edges $e\in E$.
\end{proposition}
\begin{proof}
    Since $\{l(t):t\in[0,\infty)\}$ is contained in a compact subset of $\mathbb{R}_{>0}^E$, it follows from the continuity of $H$ that the set $\{H(l(t)) : t \geq 0\} \subset \mathbb{R}$ is bounded. Since $H(l(t))$ is non-increasing, the limit $\lim\limits_{t\to\infty} H(l(t))$ exists and is finite. It follows that there exists a sequence $\{t_{n}\}_{n\geq 1}$  such that $t_n\to\infty$ as $n\to\infty$ and
    $$
    -\sum_{e\in E} \widetilde{K}_{e}^{\,2}\bigl(l(t_n)\bigr)\,l_e(t_n)= \frac{d}{dt}\Big|_{t=t_n} H\bigl(l(t)\bigr)\to0,\quad n\to\infty.
    $$
    Since each term $\widetilde{K}_e^{\,2}(l(t_n))l_e(t_n)$ is non‑negative and $l(t)$ is bounded below by $c$, we have
    \begin{equation}\label{eq:K-tilde-limit}
    \lim_{n\to\infty} \bigl|\widetilde{K}_e\bigl(l(t_n)\bigr)\bigr| = 0, \quad \forall e\in E.
    \end{equation}
    By \eqref{eq:BoundnessOfSolution}, there exists a subsequence $\{t_{n_k}\}_{k\geq1}$ of $\{t_{n}\}_{n\geq1}$ and a vector $l \in [c,C]^{E}$ such that
    \begin{equation*}
            l(t_{n_k}) \to l,\quad k\to\infty.
    \end{equation*}
    It then follows from the continuity of $\widetilde{K}(\cdot)$ and \eqref{eq:K-tilde-limit} that
    $$
    \widetilde{K}(l) = 0.
    $$
    
\end{proof}

Finally, we recall the following result from \cite{FengGeHua}, which provides a sufficient condition for a generalized hyper-ideal tetrahedron to be hyper-ideal.
    
    \begin{theorem}[{\cite[Theorem 3.9]{FengGeHua}}]\label{thm:phi-1and1}
    Let $\sigma$ be a generalized hyper‑ideal tetrahedron with edge‑length vector $l = (l_1,\dots,l_6) \in \mathbb{R}^6_{>0}$. If
    $$
    l_i \le \operatorname{arccosh} 3,\quad \forall 1 \le i \le 6,
    $$
    then
    $$
    \phi_{ij}(l) \in (-1,1), \quad \forall \{i,j\}\subset\{1,2,3,4\}.
    $$
    In particular, $\sigma$ is a hyper‑ideal tetrahedron, that is, $l \in \mathcal{L}$.
    \end{theorem}

\section{Properties of the dihedral angles of a hyper-ideal tetrahedron and a priori bounds on edge lengths}\label{sec:PropertiesDihedralAngles}

In this section, we develop some estimates of the dihedral angle of a hyper-ideal tetrahedron, which are essential for our subsequent analysis.

\begin{definition}
    Let $\sigma$ be a hyper-ideal tetrahedron, and let
    $$
    \omega=(e_1^{\sigma},e_2^{\sigma},e_3^{\sigma},e_4^{\sigma},e_5^{\sigma},e_6^{\sigma})
    $$ be an ordering of the six edges in $\sigma$. We say that $\omega$ is an \emph{edge orientation} for $\sigma$ if
    \begin{enumerate}
        \item[(1).] For each $1 \le i \le 6$, $\{e_i^\sigma, e_{i+3}^\sigma\}$ is a pair of opposite edges in $\sigma$, where indices are taken modulo $6$ with values in $\{1,\dots,6\}$.

        \item[(2).] There exists a vertex triangle such that each of the edges $e_1^{\sigma},e_2^{\sigma},e_3^{\sigma}$ has a nonempty intersection with it. Equivalently, the edges $e_4^{\sigma}, e_5^{\sigma},e_6^{\sigma}$ determine a hexagonal face in $\sigma$.
    \end{enumerate}
    For each $1 \le i \le 6$, we denote the length and the dihedral angle of $e_i^\sigma$ by $l_i^\sigma$ and $\alpha_i^\sigma$, respectively. The \emph{cosh-length} of $e_i^\sigma$ is denoted by $x_i^\sigma := \cosh(l_i^\sigma)$. Note that for an edge $e$ in $\sigma$, all the other four edges in $\sigma$ that are not opposite to $e$ are said to be \emph{adjacent} to $e$ in $\sigma$.
\end{definition}

Throughout this section, we consider a single generalized hyper-ideal tetrahedron $\sigma$ such that each edge of $\sigma$ has cosh-length in $[1,2]$. By \cref{thm:phi-1and1}, $\sigma$ is a hyper-ideal tetrahedron. Let $e$ be an edge in $\sigma$ and $\omega=(e_1^{\sigma},e_2^{\sigma},e_3^{\sigma},e_4^{\sigma},e_5^{\sigma},e_6^{\sigma})$ be an edge orientation of $\sigma$ with $e_1^{\sigma}=e$. For brevity, we shall write 
$$
e_i=e_i^{\sigma},\quad\ l_i=l_i^{\sigma},\quad\ x_i=x_i^{\sigma},\quad\ \alpha_i=\alpha_i^{\sigma},\quad\phi_i=\cos(\alpha_i^{\sigma}).
$$
Set $\alpha:=\alpha_1$ and $\phi:=\phi_1=\cos(\alpha_1)$. By \eqref{eq:phiFormula}, we have
\begin{equation}\label{eq:phi-x1-x6}
\phi =\frac{x_2x_3 + x_5x_6 + x_1x_2x_5 + x_1x_3x_6 - x_1^2x_4 + x_4}{\sqrt{2x_1x_2x_6 + x_1^2 + x_2^2 + x_6^2 - 1}\sqrt{2x_1x_3x_5 + x_1^2 + x_3^2 + x_5^2 - 1}}.
\end{equation}
This function $\phi: \mathbb{R}_{\ge 1}^6 \to \mathbb{R}$ is independent of the choice of edge orientation and clearly satisfies the following symmetry property.

\begin{proposition}
For all $x\in\mathbb{R}_{\geq1}^6$,
$$
\phi(x)=\phi(x_1,x_3,x_2,x_4,x_6,x_5)=\phi(x_1,x_5,x_6,x_4,x_2,x_3).
$$
\end{proposition}

Next we establish some monotone properties of the function $\phi(x)$ that will be used in next section.

\begin{lemma}\label{lem:10StrictlyDecreasingFct}
    Fix $a,c,d\in[1,2]$. Define the following functions:
    $$
    \begin{array}{ll}
        j_1(x)=\phi(x,a,d,c,2,2),\quad &j_2(x)=\phi(x,a,2,c,2,d),\\
        j_3(x)=\phi(x,a,x,c,2,d),\quad &j_4(x)=\phi(x,a,d,c,2,x),\\
        j_5(x)=\phi(x,a,x,x,2,2),\quad &j_6(x)=\phi(x,a,2,x,2,x),\\
        j_7(x)=\phi(x,x,x,c,2,2),\quad &j_8(x)=\phi(x,x,2,c,2,x),\\
        j_9(x)=\phi(x,x,x,x,2,2),\quad &j_{10}(x)=\phi(x,x,2,x,2,x).
    \end{array}
    $$
    Then, all these functions $j_k(x)$, $1\leq k\leq10$, are strictly decreasing with respect to $x\in[1,2]$.
    \end{lemma}
    
    \begin{proof}
        The derivatives of the functions $j_k(x)$, $1\leq k\leq6$, are as follows:
        $$
        \begin{aligned}
            j_1'(x)&=\frac{D_1(x)}{2 (3+a^2+4 a x+x^2)^{3/2} (3+d^2+4 d x+x^2)^{3/2}},\\
            j_2'(x)&=\frac{D_2(x)}{(7+x) \sqrt{7+8 x+x^2} (-1+a^2+d^2+2 a d x+x^2)^{3/2}},\\
            j_3'(x)&=\frac{D_3(x)}{\sqrt{3} (-1+a^2+d^2+2 a d x+x^2)^{3/2} (1+2 x^2)^{3/2}},\\
            j_4'(x)&=\frac{D_4(x)}{(1+a)^{1/2}(3+d^2+4 d x+x^2)^{3/2} (-1+a+2 x^2)^{3/2}},\\
            j_5'(x)&=\frac{D_5(x)}{\sqrt{3}\left(3+a^2+4ax+x^2\right)^{3/2}\left(1+2x^2\right)^{3/2}},\\
            j_6'(x)&=\frac{D_6(x)}{(7+x)\sqrt{(1+a)(7+8x+x^2)}\big(-1+a+2x^2\big)^{3/2}}.
        \end{aligned}
        $$
        where
        $$
        \begin{aligned}
        D_1(x) =& 4(a+d)\big(3a^2+3d^2 -7ad- acd - 3(1+c)\big)-2ad(46+16c-3d^2)x \\
            &- 2\Big(c(24+7d^2) +a^2(7c+8d^2+2cd^2-8)-3a^3d - 8(d^2-3) \Big)x \\
            &- 12(1+c)(4d + a^2d + a(4+d^2))x^2 \\
            &- 2\big(a^2(4+c) + 2a(5+8c)d + 4(2+d^2) + c(8+d^2)\big)x^3 \\
            &- 4(1+c)(a+d)x^4,\\
        D_2(x) =& 6a^3 - 4a^2(c+2d) - 2(2c-3d)(-1+d^2) - a(6+7cd+8d^2) \\
            &- \big(14d + a^2(11c+4d) + c(11d^2-4) + a(14+9cd+4d^2)\big)x \\
            &- \big(10d + 15acd + a^2(c+2d) + c(4+d^2) + 2a(5+d^2)\big)x^2 \\
            &- \big(2(2c+d) + a(2+cd)\big)x^3,\\
        \end{aligned}
        $$
        \begin{equation*}
            \begin{aligned}
            D_3(x) =& a(3a^2+d^2- cd-3) - \big(c(4a^2+4d^2-3) + d(2d^2-a^2)\big)x \\
            &- 9ad(c+d)x^2 - \Big(d(9+4a^2-2d^2) +c(3+2a^2+2d^2)\Big)x^3 \\
            &- 2a(3+cd-d^2)x^4,\\
            D_4(x) =& 2(-1+a)(3+3a - cd + d^2)- \Big( ac(7+2d^2) + d(2(5+d^2)-a-3a^2)-c\Big)x \\
            &- 6(1+a)d(c+d)x^2- \big(d(3+11a-2d^2) + c(9+a+2d^2)\big)x^3 \\
            &- 4(1+a+cd-d^2)x^4,\\
        \end{aligned}
        \end{equation*}
        \begin{equation*}
            \begin{aligned}
            D_5(x)=&3+a+a^2+3 a^3-2 (8-a-a^2) x-3 (3+12 a+a^2) x^2-2 (1+7 a+4 a^2) x^3\\
            &-2 (8-a+2 a^2) x^4-12 a x^5-2 x^6,\\
            D_6(x)=&-21+15 a+6 a^2-(19+9 a) x-(36 a-16) x^2-(6 a-6) x^3-22 x^4-2 x^5.\\
        \end{aligned}
        \end{equation*}
        Since $a,c,d,x\in[1,2]$, we obtain that
        $$
        \begin{array}{ll}
            D_1(x)\leq160-240x^2<0,&  D_2(x)\leq 33-63x<0,\\
            D_3(x)\leq 24-x-18x^2-12x^3<0,\qquad &D_4(x)\leq24-3x-24x^2-18x^3<0,\\
            D_5(x)\leq 33-2x-48x^2<0,& D_6(x)\leq 33-28x-20x^2-22x^4<0.
        \end{array}
        $$

        For the remaining four functions, a direct computation yields
        \begin{equation*}
            \begin{aligned}
            j_7'(x)&= -\frac{2(1+c)x}{(1+2x^2)^2}<0,\\
            j_8'(x)&= -\frac{(1+x)^2\big(4(7+x^3-3x)+c(18x+3x^2-4-x^3)\big)}{(7+x)\left((1+x)^2(-1+2x)\right)^{3/2}\sqrt{7+8x+x^2}}<0,\\
            j_9'(x)&= -\frac{-1+6x+5x^2+2x^4}{3(1+2x^2)^2}<0,\\
            j_{10}'(x)&=- \frac{(1+x)^2\big(35+9x^2+20x^3+x^4-29x\big)}{(7+x)\big((1+x)^2(-1+2x)\big)^{3/2}\sqrt{7+8x+x^2}}<0.
            \end{aligned}
        \end{equation*}
        
    \end{proof}

\begin{lemma}\label{lem:phiLessThan}
    Suppose that $x\in[1,2]^6$. Then 
    $$
    \dfrac{\partial \phi}{\partial x_j}(x)\geq0\quad
    \text{for each $j\in\{2,3,5,6\}$\quad and}\quad \dfrac{\partial \phi}{\partial x_4}(x)\leq0.
    $$ In particular, we have
    $$
    \phi(x_1,1,1,2,1,1)\leq\phi(x_1,x_2,x_3,x_4,x_5,x_6)\leq\phi(x_1,2,2,1,2,2).
    $$
\end{lemma}
\begin{proof}
    The inequality $\dfrac{\partial \phi}{\partial x_4}(x)\leq0$ is clear from the definition of $\phi$. For the remaining partial derivatives, a direct computation gives
    \begin{equation}\label{eq:PhiDerivative}
    \begin{cases}
    \dfrac{\partial \phi}{\partial x_2} = A_0[(x_1x_4 + x_2x_5 - x_3x_6)x_6 + x_3 + x_1x_5 + x_2x_4], \\
    \dfrac{\partial \phi}{\partial x_5} = A_1[(x_1x_4 + x_2x_5 - x_3x_6)x_3 + x_6 + x_1x_2 + x_5x_4], \\
    \dfrac{\partial \phi}{\partial x_3} = A_1[(x_1x_4 - x_2x_5 + x_3x_6)x_5 + x_2 + x_1x_6 + x_3x_4], \\
    \dfrac{\partial \phi}{\partial x_6} = A_0[(x_1x_4 - x_2x_5 + x_3x_6)x_2 + x_5 + x_1x_3 + x_6x_4],
    \end{cases}
    \end{equation}
    where
    \begin{align*}
    A_0 &= (x_1^2 - 1)(2x_1x_2x_6 + x_1^2 + x_2^2 + x_6^2 - 1)^{-3/2}(2x_1x_3x_5 + x_1^2 + x_3^2 + x_5^2 - 1)^{-1/2} \geq 0, \\
    A_1 &= (x_1^2 - 1)(2x_1x_2x_6 + x_1^2 + x_2^2 + x_6^2 - 1)^{-1/2}(2x_1x_3x_5 + x_1^2 + x_3^2 + x_5^2 - 1)^{-3/2} \geq 0.
    \end{align*}
    For $j=2$, it follows from \eqref{eq:PhiDerivative} that
    $$
    \begin{aligned}
        \dfrac{\partial \phi}{\partial x_2} \geq& A_0[(2 - x_3x_6)x_6 + x_3 + 2]\\=&A_0[x_3(1-x_6^2)+2x_6+2]\\
        \geq &A_0[2(1-x_6^2)+2x_6+2]=2A_0(-x_6^2+x_6+2)\geq0.
    \end{aligned}
    $$
    The other three cases where $j=3,5,6$ follow similarly.
\end{proof}

We first derive an a priori upper bound for the length of an edge $e$ along the flow, which depends only on the valence of $e$. Note that the function 
$$
\phi(x,2,2,1,2,2)=\frac{-x^2+8x+9}{x^2+8x+7}=\frac{-x+9}{x+7}
$$
is strictly decreasing for $x\in(1,+\infty)$. For each integer $n \ge 10$, let $b_n > 1$ be the unique solution of the equation 
$$
\phi(x, 2, 2, 1,2, 2) = \cos\left( \frac{2\pi}{n} \right)
$$ 
and set $b_9=2$. It follows that 
\begin{equation}\label{eq:bnExpression}
    b_n=\frac{16}{1+\cos\left( \frac{2\pi}{n} \right)}-7<2,\quad\forall n\geq 10.
\end{equation}
The following proposition gives an a priori upper bound for the edge lengths along the combinatorial Ricci flow.

\begin{proposition}\label{prop:PrioriUpperBound}
    Let $(M, \mathcal{T})$ be a closed pseudo $3$-manifold such that $v(e)\geq 9$ for each edge $e\in E$. Suppose that $\{l(t) : t \in [0, \infty)\} \subset \mathbb{R}^E_{>0}$ is the solution of the extended Ricci flow with initial data $l^0\in \mathbb{R}^E_{>0}$ satisfying 
    $$
    l^0_e\in(0,\operatorname{arccosh}b_{v(e)}),\quad \forall e\in E.
    $$
    If there exists $t_0>0$ such that
    $$
    \sup\{l_e(t):0\leq t\leq
    t_0, e\in E\}\leq \operatorname{arccosh} 2,
    $$
    then for all $t\in[0,t_0]$ and all $e\in E$,
    $$
    l_{e}(t) \in \left(0, \operatorname{arccosh} b_{v(e)}\right].
    $$
\end{proposition}

\begin{proof}
     We only need to consider the case $v(e_1)\geq10$. Suppose, for contradiction, that there exist $t_1\in[0,t_0]$ and an edge $e_1\in E$ with $v(e_1)\geq10$ such that $l_{e_1}(t_1)>\operatorname{arccosh} b_{v(e_1)}$. Set $b_0:=\cosh(l_{e_1}(t_1))>b_{v(e_1)}$. Without loss of generality, we may assume that 
    $$
    t_1=\inf\{t\in[0,t_0]:l_{e_1}(t)=\operatorname{arccosh}(b_0)\}.
    $$
    Since $l^0_e\in(0,\operatorname{arccosh}b_{v(e)})$ for all $e\in E$, we have $t_1>0$.
     It follows from \cref{lem:phiLessThan} that at time $t_1$, for each $\hat{e}\in P^{-1}_E(e_1)$,
    $$
    \begin{aligned}
        \cos(\alpha(\hat{e}))\leq&\phi(\cosh(l_{e_1}(t_1)),2,2,1,2,2)\\
        =&\phi(b_0,2,2,1,2,2)<\phi(b_{v(e_1)},2,2,1,2,2)=\cos\left(\frac{2\pi
    }{v(e_1)}\right).
    \end{aligned}
    $$
    Therefore, 
    $$
    \widetilde{K}_{e_1}(l(t_1))=2\pi-\sum_{\hat{e}\in P^{-1}_E(e_1)}\alpha(\hat{e})<2\pi-v(e_1)\cdot\frac{2\pi}{v(e_1)}=0.
    $$
    Consequently, we have $0\leq l'_{e_1}(t_1)=\widetilde{K}_{e_1}(l(t_1))l_{e_1}(t_1)<0$, which is a contradiction.
\end{proof}

For a priori lower bounds, we consider the function 
$$
\phi(x,1,1,2,1,1)=\frac{-2x^2+2x+4}{x^2+2x+1}=\frac{-2x+4}{x+1},
$$ 
which is strictly decreasing for $x\in(1,+\infty)$. For each integer $n \ge 6$, let $\gamma_n > 1$ be the unique solution of the equation 
$$
\phi(x, 1, 1, 2, 1, 1) = \cos\left( \frac{2\pi}{n} \right).
$$ 
A straightforward computation gives
\begin{equation}\label{eq:GammanExpression}
    \gamma_n=\frac{6}{2+\cos\left( \frac{2\pi}{n} \right)}-1<2,\quad\forall n\geq 6.
\end{equation}
Parallel to \cref{prop:PrioriUpperBound}, we obtain the following a priori lower bound for the edge lengths.

\begin{proposition}\label{prop:PrioriLowerbound}
    Let $(M, \mathcal{T})$ be a closed pseudo $3$-manifold such that $v(e)\geq 6$ for each edge $e\in E$.
    Let $\{l(t) \mid t \in [0, \infty)\} \subset \mathbb{R}^E_{>0}$ be the solution of the extended Ricci flow with initial data $l^0\in (0, \operatorname{arccosh} 2)^E$ satisfying  
    $$
    l^0_e\in\left(\operatorname{arccosh}\gamma_{v(e)},\operatorname{arccosh}2\right),\quad \forall e\in E.
    $$
    If there exists $t_0>0$ such that
    $$
    \sup\{l_e(t):0\leq t\leq
    t_0, e\in E\}\leq \operatorname{arccosh} 2,
    $$
    then for all $t\in[0,t_0]$ and all $e\in E$,
    $$
    l_{e}(t) \in \left[\operatorname{arccosh}\gamma_{v(e)}, \operatorname{arccosh}2\right].
    $$
\end{proposition}
\begin{proof}
    Suppose, for contradiction, that there exists $t_1\in[0,t_0]$ and an edge $e_1\in E$ such that $l_{e_1}(t_1)<\operatorname{arccosh}\gamma_{v(e_1)}$. Set $\gamma_0:=\cosh(l_{e_1}(t_1))$, we then have $\gamma_0<\gamma_{v(e_1)}$. We may assume that 
    $$
    t_1=\inf\{t\in[0,t_0]:l_{e_1}(t)=\operatorname{arccosh}\gamma_0\}.
    $$
    Since $l_e^0\in(\operatorname{arccosh}\gamma_{v(e)},\operatorname{arccosh}2)$ for all $e\in E$, we have $t_1>0$. It follows from \cref{lem:phiLessThan} that at time $t_1$, for each $\hat{e}\in P^{-1}_E(e_1)$,
    $$
    \begin{aligned}
        \cos(\alpha(\hat{e}))\geq&\phi(\cosh(l_{e_1}(t_1)),1,1,2,1,1)\\
        =&\phi(\gamma_0,1,1,2,1,1)>\phi(\gamma_{v(e_1)},1,1,2,1,1)=\cos\left(\frac{2\pi
    }{v(e_1)}\right).
    \end{aligned}
    $$
    Therefore, 
    $$
    \widetilde{K}_{e_1}(l(t_1))=2\pi-\sum_{\hat{e}\in P^{-1}_E(e_1)}\alpha(\hat{e})>2\pi-v(e_1)\cdot\frac{2\pi}{v(e_1)}=0.
    $$
    Consequently, we have $0\geq l'_{e_1}(t_1)=\widetilde{K}_{e_1}(l(t_1))l_{e_1}(t_1)>0$, which is a contradiction.
\end{proof}

\begin{remark}
    The conclusions of the preceding two propositions remain valid if we replace $b_n$ and $\gamma_n$ with arbitrary constants $\beta_n \in [b_n, 2]$ and $\delta_n \in [1, \gamma_n]$, respectively. To facilitate the numerical computations in the subsequent section, we will hereafter work with the parameters $\beta_n$ and $\delta_n$ instead of constants $b_n$ and $\gamma_n$. For the same reason, the values of $\gamma_n$ presented in \cref{tab:EstimatedBounds} are rounded down, while those of $b_n$ are rounded up.
\end{remark}

\begin{table}[h!]
    \centering
    \caption{Numerical estimates of $\gamma_n$ and $b_n$ for $n\leq 50$}
    \label{tab:EstimatedBounds}
    \begin{tabular}{|c|cc|c|cc|}
        \hline
        $n$ & $\gamma_n$ & $b_n$ & $n$ & $\gamma_n$ & $b_n$
        \\
        \hline
        6 & 1.4000 & - & 29 & 1.0157 & 1.0947 \\
        
        7 & 1.2870 & - & 30 & 1.0146 & 1.0884 \\
        
        8 & 1.2163 & - & 31 & 1.0137 & 1.0828 \\
        
        9 & 1.1691 & 2.0000 & 32 & 1.0128 & 1.0777 \\
        
        10 & 1.1359 & 1.8446 & 33 & 1.0121 & 1.0730 \\
        
        11 & 1.1117 & 1.6898 & 34 & 1.0114 & 1.0687 \\
        
        12 & 1.0934 & 1.5744 & 35 & 1.0107 & 1.0649 \\
        
        13 & 1.0793 & 1.4861 & 36 & 1.0101 & 1.0613 \\
        
        14 & 1.0682 & 1.4168 & 37 & 1.0096 & 1.0580 \\
        
        15 & 1.0593 & 1.3615 & 38 & 1.0091 & 1.0550 \\
        
        16 & 1.0520 & 1.3166 & 39 & 1.0086 & 1.0522 \\
        
        17 & 1.0460 & 1.2796 & 40 & 1.0082 & 1.0496 \\
        
        18 & 1.0410 & 1.2488 & 41 & 1.0078 & 1.0472 \\
        
        19 & 1.0367 & 1.2228 & 42 & 1.0074 & 1.0450 \\
        
        20 & 1.0331 & 1.2007 & 43 & 1.0071 & 1.0429 \\
        
        21 & 1.0300 & 1.1818 & 44 & 1.0068 & 1.0410 \\
        
        22 & 1.0273 & 1.1654 & 45 & 1.0065 & 1.0392 \\
        
        23 & 1.0250 & 1.1512 & 46 & 1.0062 & 1.0375 \\
        
        24 & 1.0229 & 1.1387 & 47 & 1.0059 & 1.0359 \\
        
        25 & 1.0211 & 1.1277 & 48 & 1.0057 & 1.0344 \\
        
        26 & 1.0195 & 1.1180 & 49 & 1.0054 & 1.0330 \\
        
        27 & 1.0181 & 1.1093 & 50 & 1.0052 & 1.0317 \\
        
        28 & 1.0168 & 1.1016 &  &  &  \\
    \hline
    \end{tabular}
\end{table}

The following theorem is a direct consequence of \cref{prop:PrioriUpperBound} and \cref{prop:PrioriLowerbound}.
\begin{theorem}\label{thm:Lowerbound}
    Let $(M, \mathcal{T})$ be a closed pseudo $3$-manifold such that $v(e) \ge 9$ for every edge $e \in E$.
    Let $\{l(t) \mid t \in [0, \infty)\} \subset \mathbb{R}^E_{>0}$ be the solution of the extended Ricci flow with initial data $l^0\in (0, \operatorname{arccosh} 2)^E$ satisfying  
    $$
    l^0_e\in\left(\operatorname{arccosh}\gamma_{v(e)},\operatorname{arccosh}b_{v(e)}\right)\quad \forall e\in E.
    $$
    If there exists $t_0>0$ such that
    $$
    \sup\{l_e(t):0\leq t\leq
    t_0, e\in E\}\leq \operatorname{arccosh} 2,
    $$
    then for all $t\in[0,t_0]$ and all $e\in E$,
    $$
    l_{e}(t) \in \left[\operatorname{arccosh}\gamma_{v(e)}, \operatorname{arccosh}b_{v(e)}\right]\subset(0,\operatorname{arccosh}2].
    $$
\end{theorem}

\section{Upper bounds of lengths}\label{sec:UpperBounds}

We begin by outlining the principal strategy. Since $\phi(2,2,2,1,2,2)>\cos\left(\frac{2\pi}{9}\right)$, a direct upper bound for the length of a single edge $e$ of valence $9$ is unavailable. Thus, we must rely on the a priori bounds derived in \cref{sec:PropertiesDihedralAngles} to initiate a bootstrap argument for obtaining tighter upper bounds on the edge lengths. Let $e$ be an edge in a tetrahedron $\sigma$. By \cref{lem:phiLessThan}, the dihedral angle $\alpha_e$ at $e$ in $\sigma$ increases as the lengths of its adjacent edges decrease, and it decreases as the length of its opposite edge decreases. Therefore, sharper upper bounds for the adjacent edge lengths and sharper lower bounds for the opposite edge length yield a sharper lower bound for the dihedral angle at $e$. By the same argument as in \cref{prop:PrioriUpperBound}, this in turn implies a tighter upper bound for the length of the edge $e$. If there exists an edge $e'$ adjacent to $e$ in $\sigma$ with valence $v(e')>9$, then we have
$$
\cosh(l_{e'}(t))\leq b_{v(e')}<2,\quad\forall t\leq t_0
$$
by \cref{prop:PrioriUpperBound}, which allows the upper bound $b_{v(e)}$ on the length of $e$ in $\sigma$ to be further tightened. If, conversely, all four edges adjacent to $e$ have valence $9$, we apply the lower bound $\gamma_9$ to establish sharper upper bounds for these four adjacent edges, which subsequently allows us to derive a tighter upper bound for $e$ itself. These techniques are systematically applied throughout this section.

Throughout this section, we assume that $(M, \mathcal{T})$ is a closed pseudo $3$-manifold such that $v(e) \ge 9$ for every edge $e \in E$, and that 
$$
\{l(t) \mid t \in [0, \infty)\} \subset \mathbb{R}^E_{>0}
$$ 
is the unique solution of the extended Ricci flow with initial data $l^0\in \mathbb{R}_{> 0}^E$ satisfying
    $$
    l^0_e\in\left(\operatorname{arccosh}\gamma_{v(e)},\operatorname{arccosh}\left(\min\{b_{v(e)},1.9\}\right)\right),\quad \forall e\in E.
    $$
Moreover, we fix $t_0>0$ and assume that 
    $$
    \sup\{l_e(t):0\leq t\leq
            t_0, e\in E\}\leq \operatorname{arccosh} 2.
    $$
Then, for all $t\in[0,t_0]$, we have $l(t)\in \mathcal{L}(M,\mathcal{T})$ by \cref{thm:phi-1and1}. Consequently, $\widetilde{K}(l(t))=K(l(t))$ for all $t\in[0,t_0]$.
  
Let $e\in E(\mathscr{T})$ be an edge contained in some tetrahedron $\sigma\in T(\mathscr{T})$. Choose an edge orientation $(e_1^{\sigma},e_2^{\sigma},e_3^{\sigma},e_4^{\sigma},e_5^{\sigma},e_6^{\sigma})$ of $\sigma$ with $e_1^{\sigma}={e}$. Further denote $\check{e}:=P_{E}(e)$ and $\check\sigma:=P_{T}(\sigma)$. For each $i\in\{1,2,3,4,5,6\}$ and $t\geq0$, let $l_{i}^{\sigma}(t)$ denote the length of the edge $P_{E}(e_i^{\sigma})$ at time $t$ and set $x_i^{\sigma}(t):=\cosh(l_{i}^{\sigma}(t))$. The valence of $e$ is denoted by $v(e):=v(\check{e})$.
Furthermore, for each $t_1\in(0,t_0]$, we define
$$
b_M(t'):=\sup\left\{\min\{x_{e}(t),x_{e'}(t)\}:e,e' \text{ are opposite edges in some } \sigma\in T(\mathscr{T}),t\in[0,t']\right\}.
$$

We now proceed to establish upper bounds for the edge lengths along the flow, focusing primarily on edges without a priori upper bounds, i.e., those edges having valence $9$. Since establishing a direct upper bound for a valence $9$ edge is difficult, we instead focus on the minimum cosh-length, $\min\{x_e(t),x_{e'}(t)\}$, of a pair of opposite edges within a tetrahedron. The following lemma illustrates how the presence of an edge $e_4^{\sigma}$ with valence $n \geq 10$ constrains the minimum lengths of its adjacent opposite edge pairs in $\sigma$.

\begin{lemma}\label{lem:dn-Define}
     Suppose that $\sigma$ is a tetrahedron in $T(\mathscr{T})$ and $\omega=(e_1^{\sigma},e_2^{\sigma},e_3^{\sigma},e_4^{\sigma},e_5^{\sigma},e_6^{\sigma})$ is an edge orientation of $\sigma$. Assume that the edge $e_4^{\sigma}$ opposite to $e_1^{\sigma}$ in $\sigma$ has valence $v(e_4^{\sigma}) = n \ge 10$. Fix $t'\in(0,t_0]$, $b\in[b_M(t'),2]$ and $\beta_n\in[b_n,2]$, where $b_n$ is defined as in \eqref{eq:bnExpression}. Define the function
    $$
    \begin{aligned}
        h_1(x,\beta_n,b):=&\arccos(\max\{\phi(x,\beta_n,x,x,2,2),\phi(x,\beta_n,2,x,2,x)\})\\+&\arccos(\max\{\phi(x,\beta_n,x,1,2,2),\phi(x,\beta_n,2,1,2,x)\})\\+&7\arccos(\max\{\phi(x,b,b,1,2,2),\phi(x,b,2,1,2,b)\}).
    \end{aligned}
    $$
    Further assume that $d_n\in(1,2]$ satisfies 
    \begin{equation}\label{eq:dn-condition}
        h_1(d_n,\beta_n,b)\geq2\pi\quad\text{and}\quad x_e(0)\leq d_n,\ \forall e\in E.
    \end{equation}
    Then for all $t\in[0,t']$ we have 
    $$
    \max\left\{\min\{x_2^{\sigma}(t),x_5^{\sigma}(t)\},\min\{x_3^{\sigma}(t),x_6^{\sigma}(t)\}\right\}\leq d_n.
    $$
\end{lemma}
\begin{proof}
    
    First, we note that the function $h_1(x,\beta_n,b)$ is strictly increasing in $x$. This follows from the fact that each $\phi$–term in the definition of $h_1$ is strictly decreasing with respect to $x$ by \cref{lem:10StrictlyDecreasingFct}, and consequently each corresponding $\arccos$ term is strictly increasing.
    
    Let $\Omega_{\sigma}^n$ denote the set of edge orientations $\omega=(e_1^{\sigma},e_2^{\sigma},e_3^{\sigma},e_4^{\sigma},e_5^{\sigma},e_6^{\sigma})$ of $\sigma$ such that $v(e_4^{\sigma})=n$. Clearly, $\Omega_{\sigma}^n\neq\varnothing$ if and only if there exists an edge $e\sim\sigma$ such that $v(e)=n$.
    Define
    $$
    T_n:=\{\sigma\in T(\mathscr{T}):\Omega_{\sigma}^n\neq\varnothing 
    \}.
    $$
    For each $\sigma\in T_n$, define 
    $$
      D_{\sigma}(t):=\max_{\omega\in \Omega_{\sigma}^n}\max\big\{\min\{x_2^{\sigma}(t),x_5^{\sigma}(t)\},\min\{x_3^{\sigma}(t),x_6^{\sigma}(t)\}\big\}.
    $$
    Note that for any edge orientation $\omega$ of $\sigma$, we always assume that $\omega$ is of the form $(e_1^{\sigma},e_2^{\sigma},e_3^{\sigma},e_4^{\sigma},e_5^{\sigma},e_6^{\sigma})$.
    We aim to prove
    $$
    \sup\{ D_{\sigma}(t):\sigma\in T_n,t\in[0,t']\}\leq d_n.
    $$
    
    Suppose, for contradiction, that there exists $t_1\in[0,t']$ such that
    $$
    a:=\sup_{\sigma\in T_n} D_{\sigma}(t_1)> d_n.
    $$
    Without loss of generality, we may assume that 
    $$
    t_1=\inf\Big\{t\in[0,t']:\sup_{\sigma\in T_n} D_{\sigma}(t)=a\Big\}.
    $$
    Since $x_e(0)\leq d_n$ for each $e\in E$, it follows that $t_1>0$. Then there exist $\sigma_1\in T_n$, $\omega_1=(e_1^{\sigma_1},e_2^{\sigma_1},e_3^{\sigma_1},e_4^{\sigma_1},e_5^{\sigma_1},e_6^{\sigma_1})\in\Omega_{\sigma_1}^n$ and $i\in\{2,3,5,6\}$ such that
    $$
    x_i^{\sigma_1}(t_1)=D_{\sigma_1}(t_1)=a\quad\text{and}\quad \frac{d x_i^{\sigma_1}}{dt}(t_1)\geq0,
    $$
    where $x_i^{\sigma_1}(t)$ is the cosh-length of the edge $e_{i}^{\sigma_1}$ at time $t$.

    Let $F_1$ be the unique face of $\sigma_1$ containing both $e_4^{\sigma_1}$ and $e_i^{\sigma_1}$, and let $F_2$ be the face of another tetrahedron (or possibly $\sigma_1$ itself) that is glued to $F_1$. Let $\sigma_2$ denote the tetrahedron containing $F_2$. Note that $\sigma_2 \in T_n$, and it is possible that $\sigma_2 = \sigma_1$.  Let $\omega_2=(e_1^{\sigma_2},e_2^{\sigma_2},e_3^{\sigma_2},e_4^{\sigma_2},e_5^{\sigma_2},e_6^{\sigma_2})$ be an edge orientation of $\sigma_2$ such that 
    $$
    P_E(e_4^{\sigma_2})=P_E(e_4^{\sigma_1}) \quad\text{and}\quad P_E(e_i^{\sigma_2})=P_E(e_i^{\sigma_1}).
    $$
    Then there exists a unique number $j\in\{2,3,5,6\}$ such that $(e_i^{\sigma_k},e_4^{\sigma_k},e_j^{\sigma_k},e_{i+3}^{\sigma_k},e_1^{\sigma_k},e_{j+3}^{\sigma_k})$ form an edge orientation for $\sigma_k$, $k=1,2$. It is notable that we must have $e_i^{\sigma_1}\neq e_i^{\sigma_2}$; otherwise, the valence of the edge $P_E(e_i^{\sigma_1})$ would be exactly $1$, which is a contradiction.
    Since 
    $$
    x_4^{\sigma_1}(t_1)=x_4^{\sigma_2}(t_1)\leq b_n \leq\beta_n
    $$ by \cref{prop:PrioriUpperBound} and $x_{i+3}^{\sigma_1}(t_1)\geq a$ by the choice of $i$, at time $t_1$ we have
    \begin{equation}\label{eq:h1One}
        \begin{aligned}
       \alpha(e_i^{\sigma_1})+\alpha(e_i^{\sigma_2})\geq&\arccos(\phi(x_i^{\sigma_1}(t_1),\beta_n,x_{j}^{\sigma_1}(t_1),a,2,x_{j+3}^{\sigma_1}(t_1)))\\&+\arccos(\phi(x_i^{\sigma_1}(t_1),\beta_n,x_{j}^{\sigma_2}(t_1),1,2,x_{j+3}^{\sigma_2}(t_1)))\\
        \geq&\arccos\big(\max\{\phi(x_i^{\sigma_1}(t_1),\beta_n,a,a,2,2),\phi(x_i^{\sigma_1}(t_1),\beta_n,2,a,2,a)\}\big)\\
        &+\arccos\big(\max\{\phi(x_i^{\sigma_1}(t_1),\beta_n,a,1,2,2),\phi(x_i^{\sigma_1}(t_1),\beta_n,2,1,2,a)\}\big)\\
        =&\arccos(\max\{\phi(a,\beta_n,a,a,2,2),\phi(a,\beta_n,2,a,2,a)\})\\
        &+\arccos(\max\{\phi(a,\beta_n,a,1,2,2),\phi(a,\beta_n,2,1,2,a)\}).
        \end{aligned}
    \end{equation}
    Here, the first inequality follows from \cref{lem:phiLessThan} and the second inequality holds because 
    $$
    \min\{x_{j}^{\sigma_k}(t_1),x_{j+3}^{\sigma_k}(t_1)\}\leq a,\quad k\in\{1,2\}.
    $$

    Suppose that $\sigma_3$ is a tetrahedron in $T(\mathscr{T})$ such that there exists $e_0\sim\sigma_3$ satisfying $P_E(e_0)=P_E(e_i^{\sigma_1})$. Note that for each pair of opposite edges $\{e,e'\}$ in $\sigma_3$,
    $$
    \min\{x_{e}(t),x_{e'}(t)\}\leq b_M(t')\leq b.
    $$ 
    It then follows from \cref{lem:phiLessThan} that
    \begin{equation}\label{eq:h1Two}
        \begin{aligned}
        \alpha(e_0)\geq&\arccos(\max\{\phi(x_i^{\sigma_1}(t_1),b,b,1,2,2),\phi(x_i^{\sigma_1}(t_1),b,2,1,2,b)\})\\
        =& \arccos(\max\{\phi(a,b,b,1,2,2),\phi(a,b,2,1,2,b)\}).
    \end{aligned}
    \end{equation}
    
    Set $\check{e}:=P_E(e_i^{\sigma_1})\in E$. Then $e_i^{\sigma_1}$ and $e_i^{\sigma_2}$ are two distinct elements of $P_E^{-1}(\check e)$, and there are at least seven other edges remaining. Combining \eqref{eq:h1One} and \eqref{eq:h1Two}, we obtain
    $$
    \begin{aligned}
        \sum_{e\in P_E^{-1}(\check{e})}\alpha(e)\geq \alpha(e_i^{\sigma_1})+\alpha(e_i^{\sigma_2})+7\min_{e\in P_E^{-1}(\check{e})}\alpha(e)\geq h_1(a,\beta_n,b)>h_1(d_n,\beta_n,b)\geq 2\pi
    \end{aligned}
    $$ 
    Consequently,
    $$
    K_{\check{e}}(l(t_1))=2\pi-\sum_{e\in P_E^{-1}(\check{e})}\alpha(e)<0.
    $$
    This implies
    $$
    0\leq\frac{d x_i^{\sigma_1}}{dt}(t_1)=\sinh(l_{\check{e}}(t_1)) \frac{d l_{\check{e}}}{dt}(t_1)=\sinh(l_{\check{e}}(t_1))l_{\check{e}}(t_1)K_{\check{e}}(l(t_1))<0,
    $$
    which is a contradiction.
\end{proof}

The following lemma establishes a unified upper bound for the cosine value of the dihedral angle of an edge by analyzing a dichotomy based on the valence of its opposite edge.

\begin{lemma}\label{lem: phi_n_delta_b_Define}
    Fix $n\geq10$, $t'\in(0,t_0]$, $c\in[b_M(t'),2]$, $\delta_{n-1}\in[1,\gamma_{n-1}]$. Assume that $d_n\in(1,2]$ satisfies the conditions \eqref{eq:dn-condition} for $b=c$ and $\beta_n$. Define
    $$
    \phi_{d_n,\delta_{n-1},c}(x):=\max\left\{
    \begin{aligned}&\phi(x,d_n,d_n,1,2,2),\phi(x,d_n,2,1,2,d_n),\\
    &\phi(x,c,c,\delta_{n-1},2,2),\phi(x,c,2,\delta_{n-1},2,c)
    \end{aligned}\right\}.
    $$
    Then, for each tetrahedron $\sigma\in T(\mathscr{T})$ and each edge $e_1\sim \sigma$, the dihedral angle $\alpha(e_1)$ of $\sigma$ at $e_1$ satisfies 
    $$
    \cos\bigl(\alpha(e_1)\bigr)\leq\phi_{d_n,\delta_{n-1},c}(x_{e_1}(t)),\quad \forall t\in[0,t'].
    $$
\end{lemma}
\begin{proof}
    Fix $e_1$ and let $\omega=(e_1,e_2,e_3,e_4,e_5,e_6)$ be an edge orientation of $\sigma$. We distinguish two cases.
    
    \textbf{Case 1.} The edge $e_4$ opposite to $e_1$ in $\sigma$ has valence $n'\geq n$. Then we see that $\beta_n\geq b_n\geq b_{n'}$. Consequently, $d_n$ also satisfies the condition for $d_{n'}$ in \cref{lem:dn-Define} with $\beta_{n'}:=\beta_n$. It follows that
    $$
    \max\{\min\{x_2(t),x_5(t)\},\min\{x_3(t),x_6(t)\}\}\leq d_n,\quad \forall t\in[0,t'].
    $$
    After reordering the edge orientation if necessary, we may assume that
    $$
    x_2(t),x_3(t)\leq d_n
    \qquad\text{or}\qquad
    x_2(t),x_6(t)\leq d_n.
    $$
    In this case,
    $$
    \cos\bigl(\alpha(e_1)\bigr)
    \leq
    \max\{\phi(x_1,d_n,d_n,1,2,2),\phi(x_1,d_n,2,1,2,d_n)\}.
    $$
    
    \textbf{Case 2.} The edge $e_4$ has valence strictly less than $n$.
    Then $x_4(t)\geq \delta_{n-1}$ by \cref{thm:Lowerbound}, and consequently,
    $$
    \cos\bigl(\alpha(e_1)\bigr)
    \leq
    \max\{\phi(x_1,c,c,\delta_{n-1},2,2),
    \phi(x_1,c,2,\delta_{n-1},2,c)\}.
    $$
    Combining the two cases yields the desired estimate.
\end{proof}

The following lemma provides a uniform upper bound for the minimum cosh-length, $\min\{x_e(t),x_{e'}(t)\}$, of opposite edge pairs, serving as the starting point for our bootstrap argument in \cref{prop:UpperBound}.

\begin{lemma}\label{lem: q_Define}
Fix $m\geq10$, $b\in[b_M(t_0),2]$, $\beta_m\in[b_m,2]$, and let $d_m$ satisfy the conditions \eqref{eq:dn-condition} for these choices of $b$ and $\beta_m$. Define
$$
\begin{aligned}
   h_2(x,d_{m},\delta_{m-1}):=&\arccos(\max\{\phi(x,x,x,x,2,2),\phi(x,x,2,x,2,x)\})\\
   +&8\arccos(\phi_{d_{m},\delta_{m-1},x}(x)),
\end{aligned}
$$
where $\delta_{m-1}\in[1,\gamma_{m-1}]$. Suppose that $q\in(1.9,2)$ satisfies $h_2(q,d_{m},\delta_{m-1})\geq2\pi$. Then for each tetrahedron $\sigma\in T(\mathscr{T})$ and each pair $\{e,e'\}$ of opposite edges in $\sigma$,
    $$
    \min\{x_{e}(t),x_{e'}(t)\}\leq q,\quad\forall t\in [0,t_0].
    $$
\end{lemma}

\begin{proof}
    First, we observe that the function $h_2(x,d_m,\delta_{m-1})$ is strictly increasing with respect to $x$. Indeed, each $\phi$–term appearing in the definition of $h_2$ is strictly decreasing in $x$ by
    \cref{lem:10StrictlyDecreasingFct}, and hence each corresponding $\arccos$ term is strictly increasing.
    For each $\sigma\in T(\mathscr{T})$, let $(e_1^{\sigma},e_2^{\sigma},e_3^{\sigma},e_4^{\sigma},e_5^{\sigma},e_6^{\sigma})$ be an edge orientation of $\sigma$ and define 
    $$
      H_{\sigma}(t):=\max\big\{\min\{x_i^{\sigma}(t),x_{i+3}^{\sigma}(t)\}:i=1,2,3\big\}.
    $$
    We need to show
    $$
    \sup\{ H_{\sigma}(t):\sigma\in T(\mathscr{T}),t\in[0,t_0]\}\leq q.
    $$
    
    Suppose, for contradiction, that there exists $t_1\in[0,t_0]$ such that
    $$
    a:=\sup_{\sigma\in T(\mathscr{T})} H_{\sigma}(t_1)> q.
    $$
    Without loss of generality, we may assume that 
    $$
    t_1=\inf\Big\{t\in[0,t_0]:\sup_{\sigma\in T(\mathscr{T})} H_{\sigma}(t)=a\Big\}.
    $$
    Since $x_e(0)\leq1.9<q$ for each edge $e\in E$, it follows that $t_1>0$. Then there exists $\sigma_1\in T(\mathscr{T})$ and $i\in\{1,2,3,4,5,6\}$ such that
    $$
    x_i^{\sigma_1}(t_1)=H_{\sigma_1}(t_1)=a\quad\text{and}\quad \frac{d x_i^{\sigma_1}}{dt}(t_1)\geq0,
    $$
    where $x_i^{\sigma_1}(t)$ is the cosh-length of the edge $e_i^{\sigma_1}$ at time $t$. After reordering the edge orientation if necessary, we may assume that $i=1$. Then
    \begin{equation}\label{eq:h2One}
        \begin{aligned}
       \alpha(e_1^{\sigma_1})\geq&\arccos\big(\phi\big(x_1^{\sigma_1}(t_1),x_2^{\sigma_1}(t_1),x_{3}^{\sigma_1}(t_1),a,x_5^{\sigma_1}(t_1),x_6^{\sigma_1}(t_1)\big)\big)\\
        \geq&\arccos\big(\max\{\phi(x_1^{\sigma_1}(t_1),a,a,a,2,2),\phi(x_1^{\sigma_1}(t_1),a,2,a,2,a)\}\big)\\
        =&\arccos(\max\{\phi(a,a,a,a,2,2),\phi(a,a,2,a,2,a)\}),
        \end{aligned}
    \end{equation}
    where the first inequality follows from the monotonicity of $\phi(x)$ with respect to $x_4$ and the second inequality holds because 
    $$
    \min\{x_{j}^{\sigma_1}(t_1),x_{j+3}^{\sigma_1}(t_1)\}\leq a,\quad j\in\{1,2,3\}.
    $$

    Now, let $\sigma_2$ be a tetrahedron in $T(\mathscr{T})$ such that there exists $e_0\sim\sigma_2$ satisfying $P_E(e_0)=P_E(e_1^{\sigma_1})$. Note that for any pair of opposite edges $\{e,e'\}$ in any tetrahedron $\sigma\in T(\mathscr{T})$,
    $$
    \min\{x_{e}(t),x_{e'}(t)\}\leq a,\quad\forall t\in [0,t_1].
    $$ 
    Applying \cref{lem: phi_n_delta_b_Define} with $c=a$ and $t'=t_1$, we obtain
    \begin{equation}\label{eq:h2Two}
        \begin{aligned}
        \alpha(e_0)\geq\arccos(\phi_{d_m,\delta_{m-1},a}(x_{e_0}(t_1)))= \arccos(\phi_{d_m,\delta_{m-1},a}(a)).
    \end{aligned}
    \end{equation}
    Set $\check{e}:=P_E(e_i^{\sigma_1})\in E$. Combining \eqref{eq:h2One} and \eqref{eq:h2Two}, we obtain
    $$
    \begin{aligned}
        \sum_{e\in P_E^{-1}(\check{e})}\alpha(e)\geq \alpha(e_1^{\sigma_1})+8\min_{e\in P_E^{-1}(\check{e})}\alpha(e)\geq h_2(a,d_m,\delta_{m-1})>h_2(q,d_m,\delta_{m-1})\geq 2\pi.
    \end{aligned}
    $$  
    Consequently,
    $$
    K_{\check{e}}(l(t_1))=2\pi-\sum_{e\in P_E^{-1}(\check{e})}\alpha(e)<0.
    $$
    This implies
    $$
    0\leq\frac{d x_1^{\sigma_1}}{dt}(t_1)=\sinh(l_{\check{e}}(t_1)) \frac{d l_{\check{e}}}{dt}(t_1)=\sinh(l_{\check{e}}(t_1))l_{\check{e}}(t_1)K_{\check{e}}(l(t_1))<0,
    $$
    which is a contradiction.
\end{proof}

\begin{proposition}\label{prop:UpperBound}
    Let $(M, \mathcal{T})$ be a closed pseudo $3$-manifold such that $v(e) \ge 9$ for every edge $e \in E$. 
    Let $\{l(t) \mid t \in [0, \infty)\} \subset \mathbb{R}^E_{>0}$ be the solution of the extended Ricci flow with initial data $l^0\in \mathbb{R}_{> 0}^E$ satisfying
    $$
    l^0_e\in\left(\operatorname{arccosh}\gamma_{v(e)},\operatorname{arccosh}\left(\min\{b_{v(e)},1.9\}\right)\right),\quad \forall e\in E.
    $$
    Moreover, assume that there exists $t_0>0$ such that
    $$
    \sup\{l_e(t):0\leq t\leq
    t_0, e\in E\}\leq \operatorname{arccosh} 2.
    $$
    Then, for each tetrahedron $\sigma\in T(\mathscr{T})$ and each pair $\{e,e'\}$ of opposite edges in $\sigma$, we have
    $$
    \min\{x_{e}(t),x_{e'}(t)\}\leq 1.975,\quad\forall t\in [0,t_0].
    $$
\end{proposition}
\begin{proof}
    Fix $m=21$, choose $\beta_{m} = 1.1818\geq b_m$ and $\delta_{m-1}= 1.0331\leq \gamma_{m-1}$.
    Let 
    $$
    b^{(1)}:=2,\quad d_m^{(1)}=1.9480,\quad q^{(1)}:=1.9790.
    $$ 
    Clearly, $b_M(t_0)\leq b^{(1)}$. After numerical computation, we verify that 
    $$
    h_1\big(d_m^{(1)},\beta_m,b^{(1)}\big)\geq2\pi\quad \text{and}\quad h_2\big(q^{(1)},d_m^{(1)},\delta_{m-1}\big)\geq2\pi.
    $$
    It follows from \cref{lem: q_Define} that for each tetrahedron $\sigma\in T(\mathscr{T})$ and each pair $\{e,e'\}$ of opposite edges in $\sigma$,
    $$
    \min\{x_{e}(t),x_{e'}(t)\}\leq q^{(1)},\quad\forall t\in [0,t_0].
    $$

    Next, choose
    $$
    b^{(2)}:=q^{(1)},\quad d_m^{(2)}=1.9404,\quad q^{(2)}:=1.9756.
    $$
    By the above argument, we have $b_M(t_0)\leq b^{(2)}$. Numerical computation shows that 
    $$
    h_1\big(d_m^{(2)},\beta_m,b^{(2)}\big)\geq2\pi\quad \text{and}\quad h_2\big(q^{(2)},d_m^{(2)},\delta_{m-1}\big)\geq2\pi.
    $$
    It follows again by \cref{lem: q_Define} that for each tetrahedron $\sigma\in T(\mathscr{T})$ and each pair $\{e,e'\}$ of opposite edges in $\sigma$,
    $$
    \min\{x_{e}(t),x_{e'}(t)\}\leq q^{(2)}=1.9756,\quad\forall t\in [0,t_0].
    $$
    
    Finally, choose 
    $$
    b^{(3)}:=q^{(2)},\quad d_m^{(3)}=1.9392,\quad q^{(3)}:=1.975.
    $$
    Another numerical verification shows
    $$
    h_{1}\bigl(d_m^{(3)},\beta_m,b^{(3)}\bigr)\ge2\pi,\qquad
    h_{2}\bigl(q^{(3)},d_m^{(3)},\delta_{m-1}\bigr)\ge2\pi .
    $$
    A last appeal to \cref{lem: q_Define} gives, for each tetrahedron $\sigma\in T(\mathscr{T})$ and each pair $\{e,e'\}$ of opposite edges in $\sigma$,
    $$
    \min\{x_{e}(t),x_{e'}(t)\}\leq q^{(3)}=1.975,\quad\forall t\in [0,t_0].
    $$
    which completes the proof.
\end{proof}

It follows from the above proposition that $b_M(t_0)\leq1.975$. Therefore, we may fix $t'=t_0$ and $b=1.975$ throughout the remainder of this section. Then, $d_{21}=1.9392$ satisfies the condition in \cref{lem:dn-Define} for $b=1.975$ and $\beta_{21}=1.1818$.

The following lemma establishes the upper bounds for the edge lengths that is needed in the proof of \cref{thm:Begining}. It first employs \cref{lem: phi_n_delta_b_Define} to refine the preliminary upper bound $d_n$ for the minimum lengths of adjacent opposite edge pairs established in \cref{lem:dn-Define}.  Then, in the case where the other five edges in the tetrahedron all have valence $9$, it combines $d_n$ with the lower bound $\gamma_9$ to give a clear upper bound for each of the four adjacent edges.

\begin{lemma}\label{lem:PnAndQn}
    Fix the parameters 
    $$
    b=1.975,\quad \beta_{21}=1.1818,\quad \delta_{20}=1.0331,\quad \delta_9=1.169\leq\gamma_9,\quad d_{21}=1.9392.
    $$
    Suppose that $\sigma$ is a tetrahedron in $T(\mathscr{T})$ and $\omega=(e_1^{\sigma},e_2^{\sigma},e_3^{\sigma},e_4^{\sigma},e_5^{\sigma},e_6^{\sigma})$ is an edge orientation of $\sigma$. Assume that the edge $e_4^{\sigma}$ opposite to $e_1^{\sigma}$ in $\sigma$ has valence $v(e_4^{\sigma}) = n \ge 9$.
    \begin{enumerate}
        \item[\rm(1).] Define
        $$
        \begin{aligned}
            h_3(x,\beta_n):=&\arccos(\max\{\phi(x,\beta_n,x,x,2,2),\phi(x,\beta_n,2,x,2,x)\})\\+&\arccos(\max\{\phi(x,\beta_n,x,1,2,2),\phi(x,\beta_n,2,1,2,x)\})\\+&7\arccos(\phi_{d_{21},\delta_{20},b}(x)).
        \end{aligned}
        $$
        If there exist $\beta_n\in[b_n,2]$ and $q_n\in(1.9,2]$ such that
        $$
        \text{either}\quad h_3(q_n,\beta_n)\geq2\pi\quad \text{or}\quad q_n=2,
        $$
        then for all $t\in[0,t_0]$,
        $$
        \max\big\{\min\{x_2^{\sigma}(t),x_5^{\sigma}(t)\},\min\{x_3^{\sigma}(t),x_6^{\sigma}(t)\}\big\}\leq q_n.
        $$
        \item[\rm(2).] Further assume that all the other five edges $e_i^{\sigma}$ $(i\neq4)$ in $\sigma$ have valence $9$. Define
        $$
        \begin{aligned}
            h_4(x,\beta_n,d_n):=&\arccos(\max\{\phi(x,\beta_n,d_n,\delta_9,2,x),\phi(x,\beta_n,x,\delta_9,2,d_n)\})\\
            +&\arccos(\max\{\phi(x,\beta_n,d_n,1,2,2),\phi(x,\beta_n,2,1,2,d_n)\})\\
           +&7\arccos(\phi_{d_{21},\delta_{20},b}(x)),
        \end{aligned}
        $$ 
        where $d_n$ satisfies the condition in \cref{lem:dn-Define} for $n\geq10$ and $d_n=2$ for $n=9$.
        If there exist $\beta_n\in[b_n,2]$ and $p_n\in(1.9,2]$ such that
        $$
        \text{either}\quad h_4(p_n,\beta_n,d_n)\geq2\pi\quad \text{or}\quad p_n=2,
        $$
        then for all $t\in[0,t_0]$,
        $$
        \max\{x_2^{\sigma}(t),x_3^{\sigma}(t),x_5^{\sigma}(t),x_6^{\sigma}(t)\}\leq p_n.
        $$
    \end{enumerate}
\end{lemma}

\begin{proof}
    Note that the functions $h_3(x,\beta_n)$ and $h_4(x,\beta_n,d_n)$ are both strictly increasing with respect to $x$ since each $\phi$–term appearing in the definition of $h_3$ and $h_4$ is strictly decreasing in $x$ by \cref{lem:10StrictlyDecreasingFct}. Let $\Omega_{\sigma}^n$, $T_n$ and $D_{\sigma}(t)$ be defined as in the proof of \cref{lem:dn-Define}.
    
    (1). If $q_n=2$, the conclusion follows immediately from the assumption that $x_e(t)\leq 2$ for all $t\in[0,t_0]$. Hence we may assume that $q_n<2$ and $h_3(q_n,\beta_n)\geq2\pi$.
    It suffices to prove
    $$
    \sup\{ D_{\sigma}(t):\sigma\in T_n,t\in[0,t_0]\}\leq q_n.
    $$
    Suppose, for contradiction, that there exists $t_1\in[0,t_0]$ such that
    $$
    a:=\sup_{\sigma\in T_n} D_{\sigma}(t_1)> q_n.
    $$
    Without loss of generality, we may assume that 
    $$
    t_1=\inf\Big\{t\in[0,t_0]:\sup_{\sigma\in T_n} D_{\sigma}(t)=a\Big\}.
    $$
    Since $\max_{e \in E} x_e(0) \le 1.9 < q_n$, it follows that $t_1 > 0$.
    Then there exist $\sigma_1\in T_n$, $\omega_1\in\Omega_{\sigma_1}^n$ and $i\in\{2,3,5,6\}$ such that
    $$
    x_i^{\sigma_1}(t_1)=\max\big\{\min\{x_2^{\sigma_1}(t_1),x_5^{\sigma_1}(t_1)\},\min\{x_3^{\sigma_1}(t_1),x_6^{\sigma_1}(t_1)\}\big\}=D_{\sigma_1}(t_1)=a
    $$
    and $\frac{d x_i^{\sigma_1}}{dt}(t_1)\geq0$.

    Let $F_1$ be the unique face of $\sigma_1$ containing both $e_4^{\sigma_1}$ and $e_i^{\sigma_1}$, and let $F_2$ be the face of another tetrahedron (or possibly $\sigma_1$ itself) that is glued to $F_1$. Let $\sigma_2$ denote the tetrahedron containing $F_2$ and let $\omega_2=(e_1^{\sigma_2},e_2^{\sigma_2},e_3^{\sigma_2},e_4^{\sigma_2},e_5^{\sigma_2},e_6^{\sigma_2})$ be an edge orientation of $\sigma_2$ such that 
    $$
    P_E(e_4^{\sigma_2})=P_E(e_4^{\sigma_1}) \quad\text{and}\quad P_E(e_i^{\sigma_2})=P_E(e_i^{\sigma_1}).
    $$
    By the same argument as in the proof of \cref{lem:dn-Define}, we obtain
    \begin{equation}\label{eq:h3One}
        \begin{aligned}
       \alpha(e_i^{\sigma_1})+\alpha(e_i^{\sigma_2})\geq&\arccos(\max\{\phi(a,\beta_n,a,a,2,2),\phi(a,\beta_n,2,a,2,a)\})\\
        &+\arccos(\max\{\phi(a,\beta_n,a,1,2,2),\phi(a,\beta_n,2,1,2,a)\}),
        \end{aligned}
    \end{equation}
    Now let $\sigma_3$ be one of the other $7$ tetrahedra (not necessarily distinct) in $T(\mathscr{T})$ for which there exists $e_0\sim\sigma_3$ satisfying $P_E(e_0)=P_E(e_i^{\sigma_1})$.
    Applying \cref{lem: phi_n_delta_b_Define} with $c=b$ and $t'=t_1$, we have
    \begin{equation}\label{eq:h3Two}
        \begin{aligned}
        \alpha(e_0)\geq\arccos(\phi_{d_{21},\delta_{20},b}(x_{e_0}(t_1)))=\arccos(\phi_{d_{21},\delta_{20},b}(a)).
    \end{aligned}
    \end{equation}
    Set $\check{e}:=P_E(e_i^{\sigma_1})\in E$. Combining \eqref{eq:h3One} and \eqref{eq:h3Two}, we obtain
    $$
    \begin{aligned}
        \sum_{e\in P_E^{-1}(\check{e})}\alpha(e)\geq \alpha(e_i^{\sigma_1})+\alpha(e_i^{\sigma_2})+7\min_{e\in P_E^{-1}(\check{e})}\alpha(e)\geq h_3(a,\beta_n)>h_3(q_n,\beta_n)\geq 2\pi
    \end{aligned}
    $$ 
    Consequently,
    $$
    K_{\check{e}}(l(t_1))=2\pi-\sum_{e\in P_E^{-1}(\check{e})}\alpha(e)<0.
    $$
    This implies
    $$
    0\leq\frac{d x_i^{\sigma_1}}{dt}(t_1)=\sinh(l_{\check{e}}(t_1)) \frac{d l_{\check{e}}}{dt}(t_1)=\sinh(l_{\check{e}}(t_1))l_{\check{e}}(t_1)K_{\check{e}}(l(t_1))<0,
    $$
    which is a contradiction.
    
    (2). If $p_n=2$, the conclusion again follows immediately from the assumption $x_e(t)\leq2$ for all $t\in[0,t_0]$. Hence we may assume that $p_n<2$ and $h_4(p_n,\beta_n,d_n)\geq2\pi$. Let $\sigma_1$ satisfy the condition in (2) and define  
    $$
    G_{\sigma_1}(t):=\max\{x_2^{\sigma_1}(t),x_3^{\sigma_1}(t),x_5^{\sigma_1}(t),x_6^{\sigma_1}(t)\}.
    $$
    It suffices to prove
    $$
    \sup\{ G_{\sigma_1}(t):t\in[0,t_0]\}\leq p_n.
    $$
    Suppose, for contradiction, that there exists $t_1\in[0,t_0]$ such that
    $$
    a:=G_{\sigma_1}(t_1)> p_n.
    $$
    Without loss of generality, we may assume that 
    $$
    t_1=\inf\Big\{t\in[0,t_0]: G_{\sigma_1}(t)=a\Big\}.
    $$
    Since $\max_{e \in E} x_e(0) \le 1.9 < p_n$, it follows that $t_1 > 0$.
    Then there exist $i\in\{2,3,5,6\}$ such that
    $$
    x_i^{\sigma_1}(t_1)=G_{\sigma_1}(t_1)=a\quad \text{and}\quad\frac{d x_i^{\sigma_1}}{dt}(t_1)\geq0.
    $$

    Let $F_1$ be the unique face of $\sigma_1$ containing both $e_4^{\sigma_1}$ and $e_i^{\sigma_1}$, and let $F_2$ be the face of another tetrahedron (or possibly $\sigma_1$ itself) that is glued to $F_1$. Let $\sigma_2$ denote the tetrahedron containing $F_2$ and let $\omega_2=(e_1^{\sigma_2},e_2^{\sigma_2},e_3^{\sigma_2},e_4^{\sigma_2},e_5^{\sigma_2},e_6^{\sigma_2})$ be an edge orientation of $\sigma_2$ such that 
    $$
    P_E(e_4^{\sigma_2})=P_E(e_4^{\sigma_1}) \quad\text{and}\quad P_E(e_i^{\sigma_2})=P_E(e_i^{\sigma_1}).
    $$
    Then there exists a unique number $j\in\{2,3,5,6\}$ such that $(e_i^{\sigma_k},e_4^{\sigma_k},e_j^{\sigma_k},e_{i+3}^{\sigma_k},e_1^{\sigma_k},e_{j+3}^{\sigma_k})$ forms an edge orientation for $\sigma_k$, $k=1,2$.
    By \cref{thm:Lowerbound}, we have 
    $$
    x_{i+3}^{\sigma_1}(t_1)\geq \gamma_9\geq \delta_9.
    $$
    Thus, at time $t_1$,
    \begin{equation}\label{eq:h4One}
         \begin{aligned}
        \alpha(e_i^{\sigma_1})+\alpha(e_i^{\sigma_2})\geq&\arccos(\phi(x_i^{\sigma_1}(t_1),x_4^{\sigma_1}(t_1),x_{j}^{\sigma_1}(t_1),\delta_9,2,x_{j+3}^{\sigma_1}(t_1)))
        \\&+\arccos(\phi(x_i^{\sigma_1}(t_1),x_4^{\sigma_2}(t_1),x_{j}^{\sigma_2}(t_1),1,2,x_{j+3}^{\sigma_2}(t_1)))\\
        \geq&\arccos(\phi(x_i^{\sigma_1}(t_1),\beta_n,x_{j}^{\sigma_1}(t_1),\delta_9,2,x_{j+3}^{\sigma_1}(t_1)))
        \\&+\arccos(\phi(x_i^{\sigma_1}(t_1),\beta_n,x_{j}^{\sigma_2}(t_1),1,2,x_{j+3}^{\sigma_2}(t_1)))\\
        \geq&\arccos(\max\{\phi(x_i^{\sigma_1}(t_1),\beta_n,d_n,\delta_9,2,a),\phi(x_i^{\sigma_1}(t_1),\beta_n,a,\delta_9,2,d_n)\})\\
        &+\arccos(\max\{\phi(x_i^{\sigma_1}(t_1),\beta_n,d_n,1,2,2),\phi(x_i^{\sigma_1}(t_1),\beta_n,2,1,2,d_n)\})\\
        =&\arccos(\max\{\phi(a,\beta_n,d_n,\delta_9,2,a),\phi(a,\beta_n,a,\delta_9,2,d_n)\})\\
        &+\arccos(\max\{\phi(a,\beta_n,d_n,1,2,2),\phi(a,\beta_n,2,1,2,d_n)\})
        \end{aligned}
    \end{equation}
    Here, the first inequality follows from \cref{lem:phiLessThan}. The second inequality holds because 
    $$
    x_4^{\sigma_2}(t_1)=x_4^{\sigma_1}(t_1)\leq b_n\leq \beta_n
    $$
    by \cref{prop:PrioriUpperBound}. For the third inequality, if $n\geq10$, then \cref{lem:dn-Define} gives
    $$
    \min\{x_j^{\sigma_k}(t_1),x_{j+3}^{\sigma_k}(t_1)\}\leq d_n,
    \qquad k=1,2.
    $$
    If $n=9$, the same assertion follows directly from the assumptions that $d_9=2$ and $x_e(t)\leq2$. Moreover, $G_{\sigma_1}(t_1)=a$ implies that $x_j^{\sigma_1}(t_1),x_{j+3}^{\sigma_1}(t_1)\leq a$, whereas for edges in $\sigma_2$ we use the trivial upper bound $2$. Together with the monotonicity in \cref{lem:phiLessThan}, these observations yield the third inequality in \eqref{eq:h4One}.
    
    Set $\check{e}:=P_E(e_i^{\sigma_1})\in E$. Combining \eqref{eq:h4One} and \eqref{eq:h3Two}, we obtain
    $$
    \begin{aligned}
        \sum_{e\in P_E^{-1}(\check{e})}\alpha(e)\geq \alpha(e_i^{\sigma_1})+\alpha(e_i^{\sigma_2})+7\min_{e\in P_E^{-1}(\check{e})}\alpha(e)\geq h_4(a,\beta_n,d_n)>h_4(p_n,\beta_n,d_n)\geq 2\pi
    \end{aligned}
    $$ 
    Consequently,
    $$
    K_{\check{e}}(l(t_1))=2\pi-\sum_{e\in P_E^{-1}(\check{e})}\alpha(e)<0.
    $$
    This implies
    $$
    0\leq\frac{d x_i^{\sigma_1}}{dt}(t_1)=\sinh(l_{\check{e}}(t_1)) \frac{d l_{\check{e}}}{dt}(t_1)=\sinh(l_{\check{e}}(t_1))l_{\check{e}}(t_1)K_{\check{e}}(l(t_1))<0,
    $$
    which is a contradiction.
\end{proof}

\section{Proof of the main theorem}\label{sec:Proof}

\begin{theorem}\label{thm:Final}
    Let $(M, \mathcal{T})$ be a closed pseudo $3$-manifold such that $v(e) \ge 9$ for every edge $e \in E$. 
    Let $\{l(t) \mid t \in [0, \infty)\} \subset \mathbb{R}^E_{>0}$ be the solution of the extended Ricci flow with initial data $l^0\in \mathbb{R}_{> 0}^E$ satisfying
    \begin{equation}\label{eq:initial_condition}
    l^0_e\in\left(\operatorname{arccosh}\gamma_{v(e)},\operatorname{arccosh}\left(\min\{b_{v(e)},1.9\}\right)\right)\quad \forall e\in E.
    \end{equation}
    Then for all $t \geq 0$ and all $e\in E$,
        $$
        l_e(t) \in \left[\operatorname{arccosh}\gamma_{v(e)}, \operatorname{arccosh} b_{v(e)}\right]\subset(0,\operatorname{arccosh}2]
        $$
    In particular, $l(t) \in \mathcal{L}(M, \mathcal{T})$ for all $t \geq 0$.
\end{theorem}

\begin{proof}
    By the condition \eqref{eq:initial_condition} on the initial data $l^0$, \cref{thm:Lowerbound} implies that
    $$
        l_e(t) \in \left[\operatorname{arccosh}\gamma_{v(e)}, \operatorname{arccosh} b_{v(e)}\right]
    $$
    as long as 
    $$
    \sup\{l_e(\tau):0\leq \tau\leq\
        t, e\in E\}\leq \operatorname{arccosh} 2.
    $$
    Therefore, it suffices to prove that $l(t)\in(0, \operatorname{arccosh} 2]^E$ for all $t>0$. Suppose, for contradiction, that
    $$
    K := \left\{t \in [0, \infty) : \max_{e \in E} l_e(t) \geq \operatorname{arccosh} 2\right\}
    $$
    is a non-empty. Set $t_1 := \inf K$, then $t_1>0$ since $l^0 \in (0, \operatorname{arccosh} 2)^E$.
    Suppose that $e_1$ is an edge in $E$ such that
    $$
    l_{e_1}(t_1) = \max_{e \in E}l_e(t_1).
    $$
    Then $l_{e_1}(t_1) =\operatorname{arccosh}2$ and
    $$
    l_{e_1}(t) \leq \max_{e \in E} l_e(t) < \operatorname{arccosh} 2, \quad \forall t < t_1.
    $$
    This implies that
    $$
    l_{e_1}'(t_1) \geq 0.
    $$

    We claim that the dihedral angle $\alpha(\hat{e})>\frac{2\pi}{v(e_1)}$ for each $\hat{e}\in P^{-1}_E(e_1)$. If this claim holds, then
    $$
    0\leq l_{e_1}'(t_1)=K_{e_1}(l(t_1))l_{e_1}(t_1)=\left(2\pi-\sum_{\hat{e}\in P^{-1}_E(e_1)}\alpha(\hat{e})\right)l_{e_1}(t_1)<0,
    $$
    which is a contradiction and thus completes the proof.
    We now proceed to prove the claim.
    
    \textbf{Case 1.} $v(e_1)\geq10$. It follows from \cref{lem:phiLessThan} that
    $$
    \cos{\alpha(\hat{e})}\leq \phi(2,2,2,1,2,2)<\cos\left(\frac{2\pi}{10}\right)\leq\cos\left(\frac{2\pi}{v(e_1)}\right).
    $$
    
    \textbf{Case 2.} $v(e_1)=9$. Suppose that $\hat{e}\sim\sigma\in T(\mathscr{T})$. Let $(e_1^{\sigma},e_2^{\sigma},e_3^{\sigma},e_{4}^{\sigma},e_5^{\sigma},e_{6}^{\sigma})$ be an edge orientation of $\sigma$ such that $e_1^{\sigma}=\hat{e}$.

    \begin{enumerate}
        \item[]\textbf{Case 2.1.} There exists $i\in\{2,3,5,6\}$ such that the edge $P_E(e_i^{\sigma})\in E$ has valence $k\geq10$. After reordering the edge orientation if necessary, we may assume that $i=2$.
        If $k\geq11$, then by \cref{tab:EstimatedBounds} we can choose $\beta_k= 1.7\geq b_k$. By \cref{prop:PrioriUpperBound} and \cref{lem:phiLessThan},
        $$
        \cos(\alpha(e_1^{\sigma}))\leq\phi(2,\beta_k,2,1,2,2)<\cos\left(\frac{2\pi}{9}\right).
        $$
        Thus, we only need to check the case $k=10$. Now fix $\beta_{10}=1.845\geq b_{10}$, $b=1.975$ and set $n=v(P_E(e_4^{\sigma}))$
        \begin{enumerate}
            \item[(1).] If $n<24$, set $\delta_n=1.024\leq \gamma_n$. By \cref{lem:phiLessThan},
            $$
            \cos(\alpha(e_1^{\sigma}))\leq\max\{\phi(2,\beta_{10},b,\delta_n,2,2),\phi(2,\beta_{10},2,\delta_n,2,b)\}<\cos\left(\frac{2\pi}{9}\right).
            $$
            \item[(2).] If $n\geq 24$, set $\beta_n=1.139\geq b_n$. It follows that $q_n^0=1.924$ satisfies the condition of $q_n$ in \cref{lem:PnAndQn} for $\beta_n$ and then
            $$
            \cos(\alpha(e_1^{\sigma}))\leq\max\big\{\phi\big(2,\beta_{10},q_n^0,1,2,2\big),\phi\big(2,\beta_{10},2,1,2,q_n^0\big)\big\}<\cos\left(\frac{2\pi}{9}\right).
            $$
        \end{enumerate}

        \item[]\textbf{Case 2.2.} All four edges $e_i^{\sigma}$, $i\in\{2,3,5,6\}$ in $\sigma$ have valence $9$. Fix the parameters
        \begin{equation}\label{eq:parameters}
            b=1.975,\quad\beta_{21} = 1.1818,\quad d_{21}=1.9392,\quad \delta_{20} = 1.0331,\quad \delta_9=1.169.
        \end{equation}
        Then $b\geq b_M(t_0)$ by \cref{prop:UpperBound}, $\beta_{21}\geq b_{21}$, $\delta_{20}\leq \gamma_{20}$, $\delta_9\leq \gamma_9$, and $d_{21}$ satisfies the condition in \cref{lem:dn-Define} for $b$ and $\beta_{21}$. 
        
        For each $n\geq9$, we set the value of $\beta_n$, $\delta_n$, $d_n$, $q_n$, $p_n$ as in the \cref{tab:parameters}. Note that we slightly abuse notation here: the values of $\beta_{21}$, $d_{21}$, $\delta_{20}$, and $\delta_9$ presented in \cref{tab:parameters} do not match those specified in \eqref{eq:parameters}. The values in \eqref{eq:parameters} are used merely to specify the fixed parameters needed to compute $p_n$ and $q_n$ in \cref{lem:PnAndQn}, while \cref{tab:parameters} records the actual parameters employed in the final numerical verification.
        After numerical verification, we observe that $\beta_n\geq b_n$, $\delta_n\leq \gamma_n$, $d_n$ satisfies the condition in \cref{lem:dn-Define}, and $q_n$, $p_n$ satisfy the condition in \cref{lem:PnAndQn} for $b$ and $\beta_n$. It then follows that
        $$
        \cos{\alpha(\hat{e})}\leq\max\{\phi(2,q_n,q_n,\delta_n,p_n,p_n),\phi(2,q_n,p_n,\delta_n,p_n,q_n)\}<\cos\left(\frac{2\pi}{9}\right).
        $$
        \end{enumerate}
    
        Therefore, we conclude that $l(t) \in(0, \operatorname{arccosh}2]^E$. It then follows from \cref{thm:phi-1and1} that $l(t) \in \mathcal{L}(M, \mathcal{T})$.
    
         \begin{table}[h!]
        \centering
        \caption{Selected values of $\beta_n$, $\delta_n$, $d_n$, $q_n$ and $p_n$}
        \label{tab:parameters}
        \begin{tabular}{|c|c|c|c|c|c|}
        \hline
        valence & $\beta_n$ & $\delta_n$ & $d_n$ & $q_n$ & $p_n$ 
        \\
        \hline
        $n\geq29$ & $ 1.100$ & $ 1$ & 1.9333 & 1.9206 & 1.9638
        \\
        \hline
        $20 \leq n \leq 28$ & $1.201$ & $ 1.0168$ & 1.9403 & 1.9275 & 1.9711
        \\
        \hline
        $15 \leq n \leq 19$ & $ 1.362$ & $ 1.0367$ & 1.9513 & 1.9383 & 1.9825
        \\
        \hline
        $12 \leq n \leq 14$ & $ 1.575$ & $ 1.0682$ & 1.9655 & 1.9521 & 1.9970
        \\
        \hline
        $9 \leq n \leq 11$ & $ 2 $ & $ 1.1117$ & 2 & 2 & 2 \\
        \hline
        \end{tabular}
        \end{table}
    \end{proof}

\begin{proof}[Proof of \cref{thm:MainTheorem}]
        It follows immediately from \cref{prop:BoundnessToExistence} and \cref{thm:Final} that there exists a generalized hyper‑ideal metric $l\in(0,\operatorname{arccosh}2]^E$ on $(M,\mathcal{T})$ such that 
        $$
        l_e \in \left[\operatorname{arccosh}\gamma_{v(e)}, \operatorname{arccosh} b_{v(e)}\right]\quad \text{and}\quad\widetilde{K}_e(l)=0
        $$
         for all $e\in E$. By \cref{thm:phi-1and1}, $l$ is in fact a zero-curvature hyper‑ideal metric. Then by \cref{thm:convergence-equivalence}, the extended Ricci flow \eqref{eq:ExtendedRicciFlow} converges exponentially fast to $l$. For the uniqueness of the zero-curvature hyper‑ideal metric on $(M,\mathcal{T})$, see Luo and Yang’s  rigidity theorem \cite[Theorem 1.2]{LuoYang}; see also \cite[Theorem 4.10]{FengGeHua} for an alternative proof.
\end{proof}

\textbf{Acknowledgements.} 
I would like to express my sincere gratitude to my supervisor, Professor Bobo Hua, for his continuous support and invaluable guidance throughout this research. I am also grateful to Soichiro Uemura for his helpful comments, which greatly improved the presentation of this manuscript.
\bibliographystyle{plain}
\bibliography{main}

@article{KeGethurstons,
      title={Combinatorial Ricci Flow and Thurston's Triangulation Conjecture}, 
      author={Feng Ke and Ge Huabin},
      year={2025},
      journal={arXiv preprint arXiv.2502.06497}
}

@article {FengGeHua,
    AUTHOR = {Feng, Ke and Ge, Huabin and Hua, Bobo},
     TITLE = {Combinatorial {R}icci flows and the hyperbolization of a class
              of compact 3-manifolds},
   JOURNAL = {Geom. Topol.},
  FJOURNAL = {Geometry \& Topology},
    VOLUME = {26},
      YEAR = {2022},
    NUMBER = {3},
     PAGES = {1349--1384},
      ISSN = {1465-3060,1364-0380},
   MRCLASS = {57M50 (05E45 53E20 57K32 57Q15)},
  MRNUMBER = {4466650},
MRREVIEWER = {Thilo\ Kuessner},
       DOI = {10.2140/gt.2022.26.1349},
       URL = {https://doi.org/10.2140/gt.2022.26.1349},
}

@article {LuoYang,
    AUTHOR = {Luo, Feng and Yang, Tian},
     TITLE = {Volume and rigidity of hyperbolic polyhedral 3-manifolds},
   JOURNAL = {J. Topol.},
  FJOURNAL = {Journal of Topology},
    VOLUME = {11},
      YEAR = {2018},
    NUMBER = {1},
     PAGES = {1--29},
      ISSN = {1753-8416,1753-8424},
   MRCLASS = {57N10 (26B25 57Q15)},
  MRNUMBER = {3747038},
MRREVIEWER = {Jonathan\ Spreer},
       DOI = {10.1112/topo.12046},
       URL = {https://doi.org/10.1112/topo.12046},
}

@article {MR1678473,
    AUTHOR = {Bonahon, Francis},
     TITLE = {A {S}chl\"afli-type formula for convex cores of hyperbolic
              {$3$}-manifolds},
   JOURNAL = {J. Differential Geom.},
  FJOURNAL = {Journal of Differential Geometry},
    VOLUME = {50},
      YEAR = {1998},
    NUMBER = {1},
     PAGES = {25--58},
      ISSN = {0022-040X,1945-743X},
   MRCLASS = {57N10 (30F45 53C22 57M50)},
  MRNUMBER = {1678473},
MRREVIEWER = {Dubravko\ Ivan\v si\'c},
       URL = {http://projecteuclid.org/euclid.jdg/1214510045},
}

@incollection {Ushijima,
    AUTHOR = {Ushijima, Akira},
     TITLE = {A volume formula for generalised hyperbolic tetrahedra},
 BOOKTITLE = {Non-{E}uclidean geometries},
    SERIES = {Math. Appl. (N. Y.)},
    VOLUME = {581},
     PAGES = {249--265},
 PUBLISHER = {Springer, New York},
      YEAR = {2006},
      ISBN = {978-0387-29554-1; 0-387-29554-2},
   MRCLASS = {52A38 (51M09)},
  MRNUMBER = {2191251},
       DOI = {10.1007/0-387-29555-0\_13},
       URL = {https://doi.org/10.1007/0-387-29555-0_13},
}

@article {Largethan6,
    AUTHOR = {Costantino, Fran\c cois and Frigerio, Roberto and Martelli,
              Bruno and Petronio, Carlo},
     TITLE = {Triangulations of 3-manifolds, hyperbolic relative
              handlebodies, and {D}ehn filling},
   JOURNAL = {Comment. Math. Helv.},
  FJOURNAL = {Commentarii Mathematici Helvetici. A Journal of the Swiss
              Mathematical Society},
    VOLUME = {82},
      YEAR = {2007},
    NUMBER = {4},
     PAGES = {903--933},
      ISSN = {0010-2571,1420-8946},
   MRCLASS = {57M50 (57M20 57M25 57N10)},
  MRNUMBER = {2341844},
       DOI = {10.4171/CMH/114},
       URL = {https://doi.org/10.4171/CMH/114},
}

@article {MR1943885,
    AUTHOR = {Bao, Xiliang and Bonahon, Francis},
     TITLE = {Hyperideal polyhedra in hyperbolic 3-space},
   JOURNAL = {Bull. Soc. Math. France},
  FJOURNAL = {Bulletin de la Soci\'et\'e{} Math\'ematique de France},
    VOLUME = {130},
      YEAR = {2002},
    NUMBER = {3},
     PAGES = {457--491},
      ISSN = {0037-9484,2102-622X},
   MRCLASS = {52A55 (51M09 52A15 52B11)},
  MRNUMBER = {1943885},
MRREVIEWER = {Raquel\ D\'iaz},
       DOI = {10.24033/bsmf.2426},
       URL = {https://doi.org/10.24033/bsmf.2426},
}

@article {MR1075164,
    AUTHOR = {Fujii, Michihiko},
     TITLE = {Hyperbolic {$3$}-manifolds with totally geodesic boundary},
   JOURNAL = {Osaka J. Math.},
  FJOURNAL = {Osaka Journal of Mathematics},
    VOLUME = {27},
      YEAR = {1990},
    NUMBER = {3},
     PAGES = {539--553},
      ISSN = {0030-6126},
   MRCLASS = {57S30 (57N10)},
  MRNUMBER = {1075164},
MRREVIEWER = {Athanase\ Papadopoulos},
       URL = {http://projecteuclid.org/euclid.ojm/1200782445},
}

@article {MR1283870,
    AUTHOR = {Rivin, Igor},
     TITLE = {Euclidean structures on simplicial surfaces and hyperbolic
              volume},
   JOURNAL = {Ann. of Math. (2)},
  FJOURNAL = {Annals of Mathematics. Second Series},
    VOLUME = {139},
      YEAR = {1994},
    NUMBER = {3},
     PAGES = {553--580},
      ISSN = {0003-486X,1939-8980},
   MRCLASS = {57M50 (57Q99)},
  MRNUMBER = {1283870},
       DOI = {10.2307/2118572},
       URL = {https://doi.org/10.2307/2118572},
}

@article {MR1985831,
    AUTHOR = {Rivin, Igor},
     TITLE = {Combinatorial optimization in geometry},
   JOURNAL = {Adv. in Appl. Math.},
  FJOURNAL = {Advances in Applied Mathematics},
    VOLUME = {31},
      YEAR = {2003},
    NUMBER = {1},
     PAGES = {242--271},
      ISSN = {0196-8858,1090-2074},
   MRCLASS = {52A55 (57M50)},
  MRNUMBER = {1985831},
       DOI = {10.1016/S0196-8858(03)00093-9},
       URL = {https://doi.org/10.1016/S0196-8858(03)00093-9},
}

@article {LuoAngleStructure,
    AUTHOR = {Luo, Feng},
     TITLE = {Volume and angle structures on 3-manifolds},
   JOURNAL = {Asian J. Math.},
  FJOURNAL = {Asian Journal of Mathematics},
    VOLUME = {11},
      YEAR = {2007},
    NUMBER = {4},
     PAGES = {555--566},
      ISSN = {1093-6106,1945-0036},
   MRCLASS = {57M50 (57M30 57N10)},
  MRNUMBER = {2402938},
MRREVIEWER = {Bruno\ P.\ Zimmermann},
       DOI = {10.4310/AJM.2007.v11.n4.a2},
       URL = {https://doi.org/10.4310/AJM.2007.v11.n4.a2},
}

@article {MR2369192,
    AUTHOR = {Rivin, Igor},
     TITLE = {Volumes of degenerating polyhedra---on a conjecture of {J}.
              {W}. {M}ilnor},
   JOURNAL = {Geom. Dedicata},
  FJOURNAL = {Geometriae Dedicata},
    VOLUME = {131},
      YEAR = {2008},
     PAGES = {73--85},
      ISSN = {0046-5755,1572-9168},
   MRCLASS = {52B11 (52B10 57M50)},
  MRNUMBER = {2369192},
MRREVIEWER = {Raquel\ D\'iaz},
       DOI = {10.1007/s10711-007-9217-x},
       URL = {https://doi.org/10.1007/s10711-007-9217-x},
}

@article {2003Combinatorial,
    AUTHOR = {Chow, Bennett and Luo, Feng},
     TITLE = {Combinatorial {R}icci flows on surfaces},
   JOURNAL = {J. Differential Geom.},
  FJOURNAL = {Journal of Differential Geometry},
    VOLUME = {63},
      YEAR = {2003},
    NUMBER = {1},
     PAGES = {97--129},
      ISSN = {0022-040X,1945-743X},
   MRCLASS = {53C44},
  MRNUMBER = {2015261},
MRREVIEWER = {Igor\ Rivin},
       URL = {http://projecteuclid.org/euclid.jdg/1080835659},
}

@article {LuoFlowWithBoundary,
    AUTHOR = {Luo, Feng},
     TITLE = {A combinatorial curvature flow for compact 3-manifolds with
              boundary},
   JOURNAL = {Electron. Res. Announc. Amer. Math. Soc.},
  FJOURNAL = {Electronic Research Announcements of the American Mathematical
              Society},
    VOLUME = {11},
      YEAR = {2005},
     PAGES = {12--20},
      ISSN = {1079-6762},
   MRCLASS = {53C44 (52A55)},
  MRNUMBER = {2122445},
MRREVIEWER = {Igor\ Rivin},
       DOI = {10.1090/S1079-6762-05-00142-3},
       URL = {https://doi.org/10.1090/S1079-6762-05-00142-3},
}

@article {MR4498159,
    AUTHOR = {Feng, Ke and Ge, Huabin and Hua, Bobo and Xu, Xu},
     TITLE = {Combinatorial {R}icci flows with applications to the
              hyperbolization of cusped 3-manifolds},
   JOURNAL = {Int. Math. Res. Not. IMRN},
  FJOURNAL = {International Mathematics Research Notices. IMRN},
      YEAR = {2022},
    NUMBER = {20},
     PAGES = {15549--15573},
      ISSN = {1073-7928,1687-0247},
   MRCLASS = {53A70 (53E20 57K32)},
  MRNUMBER = {4498159},
MRREVIEWER = {Thilo\ Kuessner},
       DOI = {10.1093/imrn/rnab081},
       URL = {https://doi.org/10.1093/imrn/rnab081},
}

@article{Frigerio01012004,
author = {Roberto Frigerio and Bruno Martelli and Carlo Petronio},
title = {Small Hyperbolic 3-Manifolds With Geodesic Boundary},
journal = {Experimental Mathematics},
volume = {13},
number = {2},
pages = {171--184},
year = {2004},
publisher = {Taylor \& Francis},
doi = {10.1080/10586458.2004.10504531},


URL = { 
    
        https://doi.org/10.1080/10586458.2004.10504531
    
    

},
eprint = { 
    
        https://doi.org/10.1080/10586458.2004.10504531
    
    

}

}

@article {MR3424676,
    AUTHOR = {Costantino, Francesco and Gu\'eritaud, Francois and van der
              Veen, Roland},
     TITLE = {On the volume conjecture for polyhedra},
   JOURNAL = {Geom. Dedicata},
  FJOURNAL = {Geometriae Dedicata},
    VOLUME = {179},
      YEAR = {2015},
     PAGES = {385--409},
      ISSN = {0046-5755,1572-9168},
   MRCLASS = {57M50 (51M10)},
  MRNUMBER = {3424676},
MRREVIEWER = {Shawn\ Rafalski},
       DOI = {10.1007/s10711-015-0086-4},
       URL = {https://doi.org/10.1007/s10711-015-0086-4},
}

@article {MR1106755,
    AUTHOR = {Colin de Verdi\`ere, Yves},
     TITLE = {Un principe variationnel pour les empilements de cercles},
   JOURNAL = {Invent. Math.},
  FJOURNAL = {Inventiones Mathematicae},
    VOLUME = {104},
      YEAR = {1991},
    NUMBER = {3},
     PAGES = {655--669},
      ISSN = {0020-9910,1432-1297},
   MRCLASS = {57M50 (30F99 51M16 52C15)},
  MRNUMBER = {1106755},
MRREVIEWER = {N.\ V.\ Ivanov},
       DOI = {10.1007/BF01245096},
       URL = {https://doi.org/10.1007/BF01245096},
}

@article {MR1815214,
    AUTHOR = {Cohn, Henry and Kenyon, Richard and Propp, James},
     TITLE = {A variational principle for domino tilings},
   JOURNAL = {J. Amer. Math. Soc.},
  FJOURNAL = {Journal of the American Mathematical Society},
    VOLUME = {14},
      YEAR = {2001},
    NUMBER = {2},
     PAGES = {297--346},
      ISSN = {0894-0347,1088-6834},
   MRCLASS = {82B41 (52C20 60C05 60D05 82B20 82B23)},
  MRNUMBER = {1815214},
MRREVIEWER = {Mihail\ N.\ Kolountzakis},
       DOI = {10.1090/S0894-0347-00-00355-6},
       URL = {https://doi.org/10.1090/S0894-0347-00-00355-6},
}

@article {MR3375525,
    AUTHOR = {Bobenko, Alexander I. and Pinkall, Ulrich and Springborn,
              Boris A.},
     TITLE = {Discrete conformal maps and ideal hyperbolic polyhedra},
   JOURNAL = {Geom. Topol.},
  FJOURNAL = {Geometry \& Topology},
    VOLUME = {19},
      YEAR = {2015},
    NUMBER = {4},
     PAGES = {2155--2215},
      ISSN = {1465-3060,1364-0380},
   MRCLASS = {52C26 (57M50)},
  MRNUMBER = {3375525},
MRREVIEWER = {Gunter\ Semmler},
       DOI = {10.2140/gt.2015.19.2155},
       URL = {https://doi.org/10.2140/gt.2015.19.2155},
}

@article{thurston1980geometry,
  title={The Geometry and Topology of Three-Manifolds},
  author={Thurston, W. P.},
  journal={\href{http://www.msri.org/gt3m}{http://www.msri.org/gt3m}},
  year={1980}
}

@article{Martelli,
  title={An introduction to geometric topology},
  author={Martelli, Bruno},
  journal={Independently published, 488 pages, 3rd Edition},
  year={2023}
}

@inbook{Thurston_1982, place={Cambridge}, series={London Mathematical Society Lecture Note Series}, title={Hyperbolic geometry and 3-manifolds}, booktitle={Low-Dimensional Topology}, publisher={Cambridge University Press}, author={Thurston, W. P.},  year={1982}, pages={9–26}, collection={London Mathematical Society Lecture Note Series}}

@article{SnapPea,
    author ={J. R. Weeks} ,
    title ={SnapPea, the computer program} ,
    journal = {Available at http://www.geom.umn.edu/software/download/snappea.html}
}

@article {MR2052949,
    AUTHOR = {Frigerio, Roberto and Petronio, Carlo},
     TITLE = {Construction and recognition of hyperbolic 3-manifolds with
              geodesic boundary},
   JOURNAL = {Trans. Amer. Math. Soc.},
  FJOURNAL = {Transactions of the American Mathematical Society},
    VOLUME = {356},
      YEAR = {2004},
    NUMBER = {8},
     PAGES = {3243--3282},
      ISSN = {0002-9947,1088-6850},
   MRCLASS = {57M50 (57M25 57N10)},
  MRNUMBER = {2052949},
MRREVIEWER = {Colin\ C.\ Adams},
       DOI = {10.1090/S0002-9947-03-03378-6},
       URL = {https://doi.org/10.1090/S0002-9947-03-03378-6},
}

@incollection {MR1208308,
    AUTHOR = {Kojima, Sadayoshi},
     TITLE = {Polyhedral decomposition of hyperbolic {$3$}-manifolds with
              totally geodesic boundary},
 BOOKTITLE = {Aspects of low-dimensional manifolds},
    SERIES = {Adv. Stud. Pure Math.},
    VOLUME = {20},
     PAGES = {93--112},
 PUBLISHER = {Kinokuniya, Tokyo},
      YEAR = {1992},
      ISBN = {4-314-10077-X},
   MRCLASS = {57M50},
  MRNUMBER = {1208308},
MRREVIEWER = {Alan\ W.\ Reid},
       DOI = {10.2969/aspm/02010093},
       URL = {https://doi.org/10.2969/aspm/02010093},
}

@misc{feng2025hyperbolization,
      title={Hyperbolization and geometric decomposition of a class of 3-manifolds}, 
      author={Ke Feng and Huabin Ge and Yunpeng Meng},
      year={2025},
      eprint={2503.07421},
      archivePrefix={arXiv},
      primaryClass={math.GT},
      url={https://arxiv.org/abs/2503.07421}, 
}

\noindent Xinrong Zhao, xrzhao24@m.fudan.edu.cn\\
\emph{School of Mathematical Sciences, Fudan University, Shanghai, 200433, P.R. China}\\[-8pt]
\end{document}